\documentclass[11pt,twoside]{article}
\usepackage{etex}
\usepackage[dvipsnames]{xcolor} % 颜色宏包调用
\usepackage{color}

\usepackage{booktabs}
\usepackage{amsthm}

\usepackage{amsfonts,amssymb,amsmath,amsxtra,url,float} % 论文初始格式设定 ams 宏包
\allowdisplaybreaks[4] % 允许多行公式强制换页
\usepackage[colorlinks,
  linkcolor=magenta, %
  anchorcolor=Periwinkle,
  citecolor=violet,
  urlcolor=blue
  ]{hyperref} % 超链颜色设定
\usepackage{enumitem}
\usepackage{geometry} 
\usepackage{rotating} % 允许rotating环境旋转段落
\usepackage{lscape} % 允许使用landscape环境进行局部旋转
\usepackage{yhmath} % 让括号变得更圆润, 但对某些版本的tex会使得大括号发生断裂
\usepackage{multirow}
\usepackage{graphicx} % 图片宏包调用, 调用此宏包后将被允许插入png 格式图片
\usepackage{subfigure} % figure环境下的补充宏包, 允许在figure环境下对自图片进行分栏
\usepackage{tikz}
\usepackage{pgfplots}
\usepackage{tikz-3dplot}
\usetikzlibrary{patterns}
\usetikzlibrary{3d,calc}
\usetikzlibrary{decorations.pathreplacing,decorations.markings}
 \tikzset{
  on each segment/.style={
    decorate,
    decoration={
      show path construction,
      moveto code={},
      lineto code={
        \path [#1]
        (\tikzinputsegmentfirst) -- (\tikzinputsegmentlast);
      },
      curveto code={
        \path [#1] (\tikzinputsegmentfirst)
        .. controls
        (\tikzinputsegmentsupporta) and (\tikzinputsegmentsupportb)
        ..
        (\tikzinputsegmentlast);
      },
      closepath code={
        \path [#1]
        (\tikzinputsegmentfirst) -- (\tikzinputsegmentlast);
      },
    },
  },
  mid arrow/.style={postaction={decorate,decoration={
        markings,
        mark=at position 0.6 with {\arrow[#1]{stealth}} % 可根据具体情况调整数字0.6
      }}},
}
\usetikzlibrary{arrows}
\usetikzlibrary{trees}

\usetikzlibrary{matrix}
\usetikzlibrary{patterns}
\usetikzlibrary{shadings} % 调用tikz作图指令包
\usepackage{fancyhdr} % 设置页眉、页脚 % see \title[]{...}
\def\headertitle{Endomorphism algebras of Gorenstein-projective modules over string algebras}
\def\fstpage{1} % 起始页码设置
\def\page{$\begin{matrix} {\color{white}0} \\ \thepage \end{matrix}$} % 页脚页码代码转换

\usepackage[all]{xy} % 交换图箭图宏包调用, 允许在xymatrix环境下进行交换图表编译
\usepackage{dsfont} % 可以调用mathds 指令
\usepackage{cite}
\usepackage{mathrsfs} % mathscr
\numberwithin{figure}{section}
\usepackage{marginnote} % 可使用边注\marginnote{}指令
\usepackage{graphicx} % 可使用\figure 环境
\usepackage{multicol} %用于实现在同一页中实现不同的分栏

\usepackage{enumitem}
\setenumerate[1]{itemsep=0pt,partopsep=0pt,parsep=\parskip,topsep=3pt}
\setitemize[1]{itemsep=0pt,partopsep=0pt,parsep=\parskip,topsep=3pt}
\setdescription{itemsep=0pt,partopsep=0pt,parsep=\parskip,topsep=3pt}
\setlist[itemize]{leftmargin=35pt}
\setlist[enumerate]{leftmargin=35pt}
\usepackage{changepage} % changepage宏包可以支持段落的双侧缩进
\newtheorem{theorem}{Theorem}[section]
\newtheorem{lemma}[theorem]{Lemma}
\newtheorem{corollary}[theorem]{Corollary}
\newtheorem{main theorem}[theorem]{Main Theorem}
\newtheorem{proposition}[theorem]{Proposition}
\newtheorem{definition}[theorem]{Definition}

\newtheorem{remark}[theorem]{Remark}
\newtheorem{example}[theorem]{Example}
\newtheorem{notation}[theorem]{Notation}

\usetikzlibrary{arrows}

\numberwithin{equation}{section}

\def\orcid{
\begin{tikzpicture}[baseline=-1mm]
\filldraw[Green!35] (0,0) circle (5pt);
\filldraw[white] (0,0) node{\tiny\textbf{iD}};
\end{tikzpicture}
}
\def\orcid{
\begin{tikzpicture}[baseline=-1mm]
\filldraw[Green!35] (0,0) circle (5pt);
\filldraw[white] (0,0) node{\tiny\textbf{iD}};
\end{tikzpicture}
}
\newcommand{\ORCID}[1]{ORCID: \href{https://orcid.org/#1}{#1}}
\newcommand{\ORCIDNOTATION}[1]{\href{https://orcid.org/#1}{\orcid}}
\def\EnglishTitle{Endomorphism algebras of Gorenstein-projective modules over string algebras}

\def\EnglishFundings{
Yu-Zhe Liu is supported by
the National Natural Science Foundation of China (Grant Nos. 12401042 and 12561008),
the Science and Technology Foundation of the Guizhou S\&T Department (Grant Nos. VZD[2026]001, ZD[2025]085 and ZK[2024]YiBan066),
and Scientific Research Foundation of Guizhou University (Grant No. [2023]16).
}
\def\FirstAuthorORICD{% 填写ORCID号
0009-0005-1110-386X}
\def\SecondAuthorORICD{% 填写ORCID号
0009-0005-2941-3436}

\def\PaperAuthorsENname{
Yu-Zhe Liu
$^{\ref{Author1}, \ref{Author3}, \ORCIDNOTATION{\SecondAuthorORICD}\ref{orcid1}}$,
Zheng Xin
$^{\ref{Author2}, \ORCIDNOTATION{\SecondAuthorORICD}\ref{orcid2},~\ref{CorrespondingAuthor}}$
}
\def\FirstEnOrgani{School of Mathematics and Statistics, Guizhou University, Guiyang 550025, Guizhou, China}
\def\SecondEnOrgani{School of Mathematics and Statistics, Shaanxi Normal University, Xi'an 710062, China}
\def\ThirdEnOrgani{Guizhou Provincial Key Laboratory of Applied Mathematics and Computing Power \& Algorithms, Guiyang, 550025, Guizhou, China}

\def\FirstEmail{\url{yzliu3@163.com}/\url{liuyz@gzu.edu.cn} (Y.-Z. Liu)}
\def\SecondEmail{\url{xinzheng1314@yeah.net} (Z. Xin)}

\newcommand{\add}{\mathsf{add}}

\def\compos{\ \lower-0.2ex\hbox{\tikz\draw (0pt, 0pt) circle (.1em);} \ }

\newcommand{\Hom}{\mathrm{Hom}} %
\newcommand{\End}{\mathrm{End}} %
\newcommand{\Ext}{\mathrm{Ext}} %
\newcommand{\rad}{\mathrm{rad}}

\newcommand{\w}[1]{\widetilde{#1}}

\newcommand{\defines}[1]{{\it\color{violet}#1}}

\title{\bf \EnglishTitle$^{\color{red}\dag}$
\footnotetext[2]{ \tiny \EnglishFundings}
}

\vspace{5mm}

\author{\PaperAuthorsENname}
\date{ }

\begin{document}

% =========================================================
% =========================================================
%    文章首页设置 2 （不必在此处填写或修改任何内容）

%\marginpar{\tiny\listofchanges}

\thispagestyle{empty}

\maketitle

\begin{enumerate}[label=\textbf{\color{red}$\ddag$}]
  \item \footnotesize
    \begin{center}
      Corresponding author
    \end{center} \label{CorrespondingAuthor}
\end{enumerate}

%\begin{multicols}{1}
%\begin{enumerate}[leftmargin=2.4cm] \footnotesize
%  \item[\orcid]
%      \ORCID{\FirstAuthorORICD}
%      \label{orcid1} % author1's orcid
%  \item[\orcid]
%      \ORCID{\SecondAuthorORICD}
%      \label{orcid2} % author2's orcid
%  \item[\orcid]
%      \ORCID{\ThirdAuthorORICD}
%      \label{orcid3} % author3's orcid
%  \item[\orcid]
%      \ORCID{\FourthAuthorORICD}
%      \label{orcid4} % author4's orcid
%\end{enumerate}
%\end{multicols}

\begin{enumerate}[leftmargin=6.6cm] \footnotesize
  \item[\orcid]
      \ORCID{\FirstAuthorORICD}
      \label{orcid1} % author1's orcid
  \item[\orcid]
      \ORCID{\SecondAuthorORICD}
      \label{orcid2} % author2's orcid
%  \item[\orcid]
%      \ORCID{\ThirdAuthorORICD}
%      \label{orcid3} % author3's orcid
\end{enumerate}

\vspace{2mm}
\begin{enumerate}[label=\textbf{\color{red}\arabic*}] \footnotesize
  \item
    \begin{center}
      \FirstEnOrgani
    \end{center}
    \label{Author1}

  \item
    \begin{center}
      \SecondEnOrgani
    \end{center}
    \label{Author2}

  \item
    \begin{center}
      \ThirdEnOrgani
    \end{center}
    \label{Author3}
%
%
%  \item
%    \begin{center}
%      \FourthEnOrgani
%
%       E-mail: \FourthEmail
%    \end{center}
%    \label{Author4}
  \item[]
    \begin{center}
      E-mail: \FirstEmail, \SecondEmail
    \end{center}
\end{enumerate}
%    文章首页设置 2 结束
% =========================================================
% =========================================================

%\definecolor{section}{rgb}{0,0,0.5}

%=========================================================
% 正文
%=========================================================

%\marginpar{\tiny\listofchanges}

\vspace{1mm}

% =========================================================
% 英文摘要
% =========================================================

\begin{adjustwidth}{1cm}{1cm}
  \noindent \footnotesize
  \textbf{Abstract}:
Let $A$ be a string algebra. We identify its Gorenstein-projective support $\tau$-tilting pairs with a distinguished class of maximal good collections and give a finite compatibility graph for them.
Under the G-condition, we construct an explicit labelled tiled-surface model for $\mathrm{End}_A(T)$ for every nonzero basic Gorenstein-projective module $T$. Under the same condition, we show that $A$ is representation-finite if and only if $\End_A(T)$ is representation-finite for every basic Gorenstein-projective module $T$.
\vspace{1mm}


%16G10. % Representations of associative Artinian rings
%16G20; % Representations of quivers and partially ordered sets
%16E05; % Syzygies, resolutions, complexes in associative algebras
%16E10 % Homological dimension in associative algebras
%%46B99; % In 46Bxx (normed linear spaces and Banach spaces; Banach lattices): none of the above, but in this section
%46H25; % Normed modules and Banach modules, topological modules
%46M40. % Inductive and projective limits in functional analysis
%16E30, % Homological functors on modules (Tor, Ext, etc.) in associative algebras
%16S70. % Representations of quivers and partially ordered sets
%18G05. % Projectives and injectives (category-theoretic aspects)
     \label{2020MSC}

\vspace{1mm}

     \label{Keywords}
\end{adjustwidth}

\newpage
% 目录
\tableofcontents

%=========================================================

\newcommand{\spacing}[1]{%
  \renewcommand{\baselinestretch}{#1}%
  \normalsize%
}

% Compatibility definitions used only by the inserted manuscript body.
\newtheorem{theoremA}{Theorem}
\renewcommand{\thetheoremA}{\Alph{theoremA}}
\newtheorem{problem}[theorem]{Problem}
\newcommand{\Fac}{\mathsf{Fac}}
\newcommand{\Sub}{\mathsf{Sub}}
\newcommand{\bpoint}{\circ}
\newcommand{\rpoint}{{\color{red}\bullet}}
\def\perpath{\mathsf{p.path}}
\definecolor{dark-green}{RGB}{14,130,2}

\section{Introduction}

String algebras form a class of special biserial algebras whose module categories admit an explicit combinatorial description.
Indecomposable modules are parametrized by strings and bands, Auslander--Reiten sequences are described by hooks and cohooks, and morphisms between string modules can be computed from overlaps of strings \cite{BR87,CB89}.
These features make string algebras a natural setting for studying homological and mutation-theoretic questions by concrete combinatorial methods.

Gentle algebras are an important subclass of string algebras.
Their Gorenstein-projective modules have a particularly simple description \cite{K15}, and their module categories admit geometric models in terms of marked surfaces and permissible curves \cite{BC21,BC21-corr}.
Related non-kissing constructions for locally gentle algebras provide a geometric framework for compatibility and mutation \cite{PPP19}; further surface models also appear in the study of algebraic hearts and related structures \cite{C26}.
These geometric descriptions are especially effective because quadratic relations can be read directly from the local configuration of the dissection.

For an arbitrary string algebra, the defining ideal may contain minimal relations of arbitrary length.
Baur and Coelho Sim\~oes overcome this problem by introducing labelled tiled surfaces, where labels record the additional monomial relations which are not visible in the underlying locally gentle tiling \cite{BC24}.
In this model, strings and bands are represented by permissible arcs and closed curves, morphisms are described by crossings and intersections, and support $\tau$-tilting pairs correspond to maximal collections satisfying the good-crossing condition.
Thus the surface model extends the geometric methods available for gentle algebras to general string algebras while retaining the information carried by long relations.

Support $\tau$-tilting theory was introduced by Adachi--Iyama--Reiten \cite{AIR14}, and its Gorenstein-projective version was studied in \cite{XZ24}.
On the other hand, Chen, Shen and Zhou classify the non-projective indecomposable Gorenstein-projective modules over a monomial algebra by perfect paths \cite{CSZ18}.
For string algebras this description was developed further by Liu and Zhang, who also construct the Cohen--Macaulay Auslander algebra \cite{LZ24}.
In the gentle case, Gorenstein-projective support $\tau$-tilting modules and their endomorphism algebras have already been studied geometrically \cite{LZH}.
For general string algebras, however, the perfect-path description of Gorenstein-projective modules and the labelled-surface description of strings and support $\tau$-tilting pairs arise from different combinatorial data.

The first purpose of this paper is to relate perfect paths to the labelled surface model.
For each perfect path \(p\), we identify \(pA\) with the string module \(M(w_p)\).
Together with the indecomposable projective modules, this realizes all indecomposable Gorenstein-projective modules as a distinguished family of permissible arcs.
The correspondence of \cite[Theorem~5.13]{BC24} then yields a geometric classification of Gorenstein-projective support \(\tau\)-tilting pairs.

These pairs also admit a finite description.
Let \(n=|Q_0|\), and form the graph \(\mathcal G_A\) whose module vertices are the indecomposable \(\tau\)-rigid Gorenstein-projective modules, together with vertices \(i^{-}\), \(i\in Q_0\).
Distinct module vertices \(X,Y\) are adjacent when \(\Hom_A(X,\tau Y)=\Hom_A(Y,\tau X)=0\); the vertex \(i^{-}\) is adjacent to \(X\) when \(Xe_i=0\), and the negative vertices are pairwise adjacent.

\begin{theoremA}[Theorem~\ref{thm:finite-compatibility-graph}]
\label{thm:intro-GP-classification}
The \(n\)-vertex cliques of \(\mathcal G_A\) correspond bijectively to the basic Gorenstein-projective support \(\tau\)-tilting pairs.
\end{theoremA}

The graph \(\mathcal G_A\) is determined by finite path calculations.
Indeed, if \((p,q)\) is a perfect pair with \(p\) perfect, then, for every \(A\)-module \(N\),
\[
\Hom_A(N,\tau(pA))=0 \quad\Longleftrightarrow\quad Ne_{t(p)}\xrightarrow{\,\cdot q\,}Ne_{t(q)} \text{ is surjective}.
\]
No Gorenstein-projective Bongartz completion is required.
On the labelled surface, the \(n\)-vertex cliques correspond to maximal good collections whose module arcs represent Gorenstein-projective modules, while the arcs in \(\mathcal P\) encode the projective part.
Notice that \(pA\) is represented by the string \(w_p\), which need not coincide with \(p\).

The second purpose is to construct the endomorphism algebras of Gorenstein-projective modules.
If \(G_A\) is a basic additive generator of \(\mathsf{Gproj}(A)\) and \(\Gamma=\End_A(G_A)=A^{\mathrm{CMA}}\), then \(\End_A(T)\cong e_T\Gamma e_T\).
Under the G-condition, \(\Gamma\) is a string algebra \cite{LZ24}.
The corner description is valid without the G-condition, whereas the string-algebra structure and the labelled-surface realization of \(\End_A(T)\) below require the G-condition.
{As} a corner is not obtained merely by taking a full subquiver, we compress paths through omitted vertices and separate the resulting relations into tile and label relations.

\begin{theoremA}[Theorem~\ref{thm:geometric-End}]
\label{thm:intro-construction}
Let \(A\) satisfy the G-condition and let \(T\) be a nonzero basic Gorenstein-projective \(A\)-module.
Compressing the paths in the perfect-arrow presentation of \(A^{\mathrm{CMA}}\) through the vertices not corresponding to summands of \(T\) {provides} a string algebra \(\mathbf{k}Q_T/J_T\) and a labelled tiled surface \((\mathcal S_T,\mathsf P_T,\mathsf L_T)\) such that
\[
\End_A(T)\cong\mathbf{k}Q_T/J_T \cong A_{\mathsf P_T,\mathsf L_T}.
\]
A path in \(Q_T\) is zero precisely when its expansion is zero in \(A^{\mathrm{CMA}}\).
\end{theoremA}

The compression construction applies to every corner of a string algebra.
For an arbitrary string algebra \(A\), graph-map bases also {provide} a finite bound-quiver presentation of \(\End_A(T)\), although its relations may be nonmonomial and may involve paths of different lengths.
Related endomorphism algebras for string almost gentle algebras are studied in \cite{LZhou24}.

Stable endomorphism algebras are considerably simpler.
For an arbitrary string algebra, the stable endomorphism algebra of a basic \(\tau\)-rigid module is gentle or zero \cite{SZ03}.
Under the G-condition, if a basic module \(T\in\mathsf{Gproj}(A)\) has \(r\) non-projective indecomposable summands, then \(\underline{\End}_A(T)\cong\mathbf{k}^{\,r}\).
By contrast, \(\End_A(T)\) may have long monomial relations under the G-condition and may have nonmonomial relations involving paths of different lengths outside it.

We finally consider representation-finiteness.
{As} a string algebra is representation-finite precisely when it has no bands \cite{BR87}, and perfect-arrow splitting preserves bands under the G-condition, we obtain the following result.

\begin{theoremA}[Theorem~\ref{thm:rep-finite-G-condition}]
\label{thm:intro-repfinite}
Let \(A\) satisfy the G-condition.
The following are equivalent:
\begin{enumerate}
  \item[{\rm(1)}] \(A\) is representation-finite;
  \item[{\rm(2)}] \(A^{\mathrm{CMA}}\) is representation-finite;
  \item[{\rm(3)}] \(\End_A(T)\) is representation-finite for every basic Gorenstein-projective module \(T\);
  \item[{\rm(4)}] \(\End_A(T)\) is representation-finite for every basic Gorenstein-projective support \(\tau\)-tilting module \(T\).
\end{enumerate}
\end{theoremA}

\section{Preliminaries}
Throughout the paper, $\mathbf{k}$ is an algebraically closed field, all algebras are finite-dimensional $\mathbf{k}$-algebras unless explicitly stated otherwise, and all modules are finitely generated right modules.
We write $\operatorname{mod}A$ for their category and $\operatorname{ind}(\mathcal{C})$ for the set of isomorphism classes of indecomposable objects of a Krull--Schmidt category $\mathcal{C}$.
A quiver $Q=(Q_0,Q_1,s,t)$ is always finite.
Here $Q_0$ is the vertex set, $Q_1$ is the arrow set, and $s,t:Q_1\to Q_0$ assign to an arrow its source and target.
A loop at $v$ is an arrow $a$ with $s(a)=t(a)=v$.

A path of length $n\geqslant1$ is a sequence $p=a_1\cdots a_n$ such that $t(a_i)=s(a_{i+1})$ for $1\leqslant i<n$.
It is a cycle if $t(a_n)=s(a_1)$.
The trivial path at $v$ is denoted by $e_v$.
We compose paths from left to right: if $t(a)=s(b)$, then $ab$ starts at $s(a)$ and ends at $t(b)$.
Composition of module homomorphisms is written in the usual right-to-left order.

Let $R_Q$ be the arrow ideal of the path algebra $\mathbf{k}Q$.
An ideal $I\subseteq\mathbf{k}Q$ is \defines{admissible} if $R_Q^m\subseteq I\subseteq R_Q^2$ for some $m\geqslant2$.
In this case $\mathbf{k}Q/I$ is finite-dimensional.

\subsection{String algebras, strings and bands}
Now we recall the definition of string algebras.

\begin{definition}[{\!\!\cite{BR87}}] \label{definition:string algebras} \rm
An algebra $A=\mathbf{k}Q/I$ is called a \defines{string algebra} if $I$ is an admissible ideal of $\mathbf{k}Q$ satisfying the following conditions:
\begin{enumerate}[label=(S\arabic*)]
  \item Each vertex in $Q_0$ is the source of at most two arrows and the target of at most two arrows.\label{S1}

  \item For each arrow $a$, there is at most one arrow $b$ with $t(a)=s(b)$ and $ab\notin I$, and at most one arrow $c$ with $t(c)=s(a)$ and $ca\notin I$.\label{S2}

  \item $I$ is generated by paths of length at least two, i.e. $A$ is monomial.
\label{S3}
\end{enumerate}
\end{definition}

\begin{definition}[{\!\!\cite{AS87}}] \rm
\label{definition:gentle algebras}
A string algebra $A=\mathbf{k}Q/I$ is called \defines{gentle} if it satisfies the following additional conditions:
\begin{enumerate}[label=(G\arabic*)]
  \item For each arrow $a$, there is at most one arrow $b'$ with $t(a)=s(b')$ and $ab'\in I$, and at most one arrow $c'$ with $t(c')=s(a)$ and $c'a\in I$.\label{G1}

  \item $I$ is generated by paths of length two.\label{G2}
\end{enumerate}
\end{definition}

A vertex $v$ satisfying \ref{S1} and for which the paths of length 2 going through $v$ satisfy \ref{S2} and \ref{G1} is called a \defines{gentle vertex}.
Otherwise, it is called a \defines{non-gentle vertex}.

\begin{definition}\label{definition:locally gentle algebras} \rm
A \defines{locally gentle algebra} is an algebra which admits a presentation $\mathbf{k}Q/I$ satisfying the conditions in Definition \ref{definition:gentle algebras} but which is not necessarily finite dimensional, i.e. there may not exist \( m \geqslant 2 \) such that \( R_Q^m \subseteq I\).
\end{definition}

\begin{example} \rm
\label{ex:string-gentle algebra}
Consider the quiver $Q$
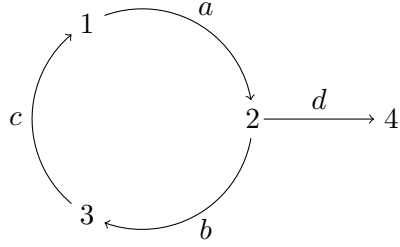
        \[\begin{tikzpicture}[scale=0.73]
            \draw (2,0) node {$2$} (-1,1.73) node{$1$} (-1,-1.73) node{$3$};
            \foreach \x in {0,120,240}
            \draw[rotate=-10+\x][->] (-1,1.73) arc(120:20:2);
            \draw[->] (2.2,0) -- (4.2,0);
            \draw (4.2,0) node[right]{$4$};
            \draw (3.2,0) node[above]{$d$};
            \draw[rotate=60+  0] (2.3,0) node{$a$};
            \draw[rotate=60+240] (2.3,0) node{$b$};
            \draw[rotate=60+120] (2.3,0) node{$c$};
        \end{tikzpicture}\]
%\begin{center}
%\begin{tikzcd}
%1 \arrow[rr, "a"] &                     & 2 \arrow[ld, "b"] \arrow[r, "d"] & 4 \\
%& 3 \arrow[lu, "c"] &                                  &
%\end{tikzcd}
%\end{center}
and the following ideals:
\begin{center}
$I_{1}=\langle bca,ad,ab\rangle$, $I_{2}=\langle ad\rangle$, and $I_{3}=\langle ab\rangle$.
\end{center}
Then $\mathbf{k}Q/I_{1}$ is a string algebra which is neither gentle nor locally gentle.
The algebra $\mathbf{k}Q/I_{2}$ is locally gentle but not gentle, and $\mathbf{k}Q/I_{3}$ is gentle.

\end{example}

Let $A = \mathbf{k}Q/I$ be a string algebra.
Every indecomposable $A$-module is isomorphic to either a string module or a band module.
We briefly recall the relevant definitions and refer to \cite{BR87} for details.
For each arrow $a \in Q_1$, we define a formal inverse $a^{-1}$ with $s(a^{-1}) = t(a)$ and $t(a^{-1}) = s(a)$.
We also set $(a^{-1})^{-1} = a$.

A \defines{string} of length $n \ge 1$ is a sequence $w = l_1 l_2 \cdots l_n$, where each $l_i$ is an arrow or a formal inverse of an arrow, satisfying the following conditions:
\begin{enumerate}
\item[(1)] $t(l_i) = s(l_{i+1})$ for all $1 \le i \le n - 1$.
\item[(2)] $l_{i+1} \neq l_i^{-1}$ for all $1 \le i \le n - 1$.
\item[(3)] No subword of $w$ or its formal inverse $w^{-1} = l_n^{-1} \cdots l_1^{-1}$ belongs to the ideal $I$.
\end{enumerate}
For each vertex $v \in Q_0$, there is also a trivial string $e_v$ of length $0$.
In the signed convention used in the surface model, $e_v$ has two signs; they are inverse to one another and determine the same simple module and the same underlying unoriented arc.
The \defines{string module} $M(w)$ is obtained by assigning one copy of $\mathbf{k}$ to each occurrence of a vertex along $w$ and the identity map to each arrow occurring in the graph of $w$.
In particular, repeated occurrences of the same vertex contribute distinct basis vectors.
Moreover, $M(w)\cong M(w^{-1})$.

We shall also use the sign convention of \cite[Definition~1.5]{BC24}.
A pair of \defines{sign functions} for $(Q,I)$ consists of maps
\[
\sigma,\varepsilon:Q_1\longrightarrow\{-1,1\}
\]
such that distinct arrows with the same source have different \(\sigma\)-signs, distinct arrows with the same target have different \(\varepsilon\)-signs, and for any pair of arrows \(a,b\) satisfying \(t(a)=s(b)\) and \(ab\notin I\), one has \(\sigma(b)=-\varepsilon(a)\).
Such sign functions exist for every string algebra.
For a nontrivial direct path $u=a_1\cdots a_r$, {let} $\sigma(u)=\sigma(a_1)$ and $\varepsilon(u)=\varepsilon(a_r)$.

A \defines{band} is a cyclic string \(b=l_1\cdots l_n\) such that every power \(b^m\) is a string and \(b\) is not a proper power of a shorter string.
We identify bands which differ by a cyclic permutation or inversion.
A band \(b\), an integer \(d\geqslant1\), and an indecomposable automorphism \(\varphi\) of \(\mathbf{k}^d\) determine a \defines{band module} \(B(b,d,\varphi)\).

\begin{proposition}[\cite{BR87}]\label{prop:band-criterion}
A string algebra is representation-finite if and only if it admits no band.
\end{proposition}

\subsection{Perfect paths and Gorenstein-projectives}

An $A$-module $G$ is \defines{Gorenstein-projective} if it occurs as a syzygy in a totally acyclic complex of finitely generated projective modules; equivalently, there is an exact complex
\[
\cdots\longrightarrow P^{-1}\longrightarrow P^0\longrightarrow P^1 \longrightarrow P^2\longrightarrow\cdots
\]
which remains exact after applying $\Hom_A(-,A)$ and such that $G\cong\operatorname{Im}(P^{-1}\to P^0)$.
We write $\mathsf{Gproj}(A)$ for the full subcategory of $\operatorname{mod}A$ consisting of the Gorenstein-projective modules.
The algebra $A$ is called \defines{CM-finite} if $\mathsf{Gproj}(A)$ has only finitely many isomorphism classes of indecomposable objects, and it is called \defines{CM-free} if every Gorenstein-projective $A$-module is projective.

\begin{definition}[{\!\!\cite[Definitions~3.3 and~3.7]{CSZ18}}]\label{def:perfect-path} \rm
Let $p,q$ be nontrivial paths with $p,q\notin I$.
The pair $(p,q)$ is a \defines{perfect pair} if:
\begin{enumerate}[label=(\arabic*), label={(P\arabic*)}]
\item $t(p)=s(q)$ and $pq\in I$; \label{P1} \item whenever $q'$ is a nontrivial path with $q'\notin I$, $s(q')=t(p)$ and $pq'\in I$, the path $q$ is a prefix of $q'$; \label{P2} \item whenever $p'$ is a nontrivial path with $p'\notin I$, $t(p')=s(q)$ and $p'q\in I$, the path $p$ is a suffix of $p'$.
\label{P3}
\end{enumerate}

A cyclic sequence $(p_1,\ldots,p_m)$ is a \defines{perfect path sequence} if $(p_i,p_{i+1})$ is a perfect pair for every $i$, with indices read modulo $m$.
A path occurring in such a sequence is a \defines{perfect path}.
We denote the set of perfect paths by $\perpath(A)$.
\end{definition}

We recall the following characterization of non-projective indecomposable Gorenstein-projective modules from \cite[Theorem~4.1]{CSZ18}; see also \cite[Theorem~3.7]{LZ24}.

\begin{theorem}\label{thm:perfect-path-GP}
Let $A=\mathbf{k}Q/I$ be a string algebra.
The assignment $p\mapsto pA$ induces a bijection from $\perpath(A)$ to the isomorphism classes of non-projective indecomposable Gorenstein-projective $A$-modules.
Consequently,
\[
\operatorname{ind}(\mathsf{Gproj}(A)) =\{e_iA\mid i\in Q_0\}\sqcup\{pA\mid p\in\perpath(A)\}.
\]
\end{theorem}

\begin{corollary}\label{cor:CM-finite}
Every string algebra is CM-finite, and
\[
\#\operatorname{ind}(\mathsf{Gproj}(A)) =|Q_0|+|\perpath(A)|.
\]
Moreover, $A$ is CM-free if and only if $\perpath(A)=\varnothing$.
\end{corollary}

\begin{proof}
There are only finitely many nonzero paths in the finite-dimensional monomial algebra $A$, so $\perpath(A)$ is finite.
The remaining assertions follow from Theorem~\ref{thm:perfect-path-GP}.
\end{proof}

\subsection{Relation cycles and the G-condition}
Throughout this paper, a relation means a minimal monomial generator of $I$.
In particular, no proper subpath of a relation belongs to $I$.

\begin{definition}[{\!\!\cite[Section~2]{LZ24}}]\label{def:relation-cycle} \rm
Let \(C=\alpha_1\cdots\alpha_n\) be a primitive oriented cyclic word in \(Q\).
A relation cycle on \(C\) is a cyclic family of minimal relations whose successive occurrences overlap along the periodic word \(C^\infty\).
More precisely, index the occurrences by \(i\in\mathbb Z\), and represent the \(i\)-th occurrence by an interval \([a_i,b_i)\).
We require
\[
a_i<a_{i+1}<b_i<b_{i+1}, \quad (a_{i+d},b_{i+d})=(a_i+n,b_i+n)
\]
for every \(i\in\mathbb Z\) and some \(d\geqslant1\).
Thus consecutive relations have a nontrivial proper overlap and advance around the cycle.
Relations may wind more than once around \(C\).
\end{definition}

\begin{definition}\label{def:full-relational-cycle} \rm
A primitive oriented cyclic word $C=\alpha_1\cdots\alpha_n$ is a \defines{full-relation oriented cycle} if $\alpha_i\alpha_{i+1}\in I$ for every $i\in\mathbb Z/n\mathbb Z$.
A relation cycle is called \defines{gentle} if all its relations are quadratic.
\end{definition}

    Every path of length two belonging to $I$ is therefore a minimal relation.
    {Therefore} a full-relation oriented cycle is a gentle relation cycle, where ``gentle'' refers only to the 2-relations.

\begin{definition}[{\!\!\cite[Section~4.1]{LZ24}}] \rm \label{def:G-condition}
A relation cycle \defines{supports a perfect path sequence} if the concatenation of the sequence is a power of its underlying cyclic word, and the successive products $p_ip_{i+1}$ occur among its relations.
A string algebra satisfies the \defines{G-condition} if every relation cycle supporting a perfect path sequence is gentle.
\end{definition}

For a perfect pair $(p,q)$, the path $pq$ is a minimal relation: a minimal relation inside $pq$ must meet both $p$ and $q$, and \ref{P2} and \ref{P3} force its final and initial points to be those of $pq$.
Consequently, the successive products in a perfect path sequence form overlapping minimal relations.
Thus, for string algebras, the G-condition is equivalently expressed as $\perpath(A)\subseteq Q_1$.
Indeed, a quadratic product $p_ip_{i+1}$ forces both factors to have length one; conversely, a sequence of perfect arrows {provides} only quadratic relations on its cycle.
Long minimal relations away from these cycles are allowed, as in Example~\ref{ex:nongentle-G}.

\subsection{Support \texorpdfstring{$\tau$}{tau}-tilting pairs}

For an $A$-module $M$, let $|M|$ be the number of isomorphism classes of indecomposable direct summands of a basic representative.
Recall from \cite{AIR14} that $M$ is \defines{$\tau$-rigid} if $\Hom_A(M,\tau M)=0$.
A pair $(M,P)$ with $P$ projective is a \defines{$\tau$-rigid pair} if $M$ is $\tau$-rigid and $\Hom_A(P,M)=0$; it is a \defines{support $\tau$-tilting pair} if, in addition, $|M|+|P|=|Q_0|$.
We call $M$ a \defines{Gorenstein-projective support $\tau$-tilting module} if it is the module part of a support $\tau$-tilting pair and $M\in\mathsf{Gproj}(A)$, following \cite{XZ24}.
The projective part $P$ is a separate component of the pair; the condition $\Hom_A(P,M)=0$ implies $\add P\cap\add M=0$.

We shall use the following maximality criterion.

\begin{proposition}[{\!\!\cite[Corollary~2.13]{AIR14}}]\label{prop:AIR-maximality}
A basic $\tau$-rigid pair is support $\tau$-tilting if and only if it is maximal with respect to adjoining indecomposable summands while preserving the $\tau$-rigid-pair conditions.
\end{proposition}

\begin{proposition}[{\!\!\cite[Proposition~2.4 (b), (c)]{AIR14}}]
\label{prop:projective-presentation-tau-rigid}
Let $P_1\xrightarrow{d_1}P_0\to M\to0$ be a minimal projective presentation, and let $N$ be an $A$-module.
Then $\Hom_A(N,\tau M)=0$ if and only if
\[
\Hom_A(P_0,N)\xrightarrow{\Hom_A(d_1,N)}\Hom_A(P_1,N)
\]
is surjective.
In particular, taking $N=M$ {provides} the criterion for $\tau$-rigidity.
\end{proposition}

\subsection{Labelled tiled surfaces and labelled tiling algebras}\label{subsection: geo-module categories}

We use the labelled tiled surfaces of \cite{BC24}, extending the gentle surface model of \cite{BC21,BC21-corr}.
The locally gentle reconstruction is that of \cite{PPP19}.

\begin{definition}\rm
A \textit{marked surface} is a pair $(\mathcal{S},\mathcal{M})$, where
\begin{enumerate}[label=(\arabic*)]
  \item $\mathcal{S}$ is a compact connected oriented surface with possibly empty boundary $\partial \mathcal{S}=\sqcup_{i=1}^{b}\partial_i \mathcal{S}$;
  \item $\mathcal{M}=\mathcal{M}_{\circ}\cup\mathcal{P}_{\circ}\cup\mathcal{M}_{\rpoint}\cup\mathcal{P}_{\rpoint}$ is a finite set of marked points on $\mathcal{S}$, represented by the symbols $\circ$ and $\rpoint$, respectively.
Each boundary component $\partial_i \mathcal{S}$ contains at least one marked point of each color from $\mathcal{M}_{\circ}\cup\mathcal{M}_{\rpoint}$, and the points $\circ$ and $\rpoint$ occur alternately.
The elements of $\mathcal{P}_{\circ}\cup\mathcal{P}_{\rpoint}$ lie in the interior of $\mathcal{S}$.
\end{enumerate}
\end{definition}

We will often simply write $\mathcal{S}$ for $(\mathcal{S},\mathcal{M})$.
For a disconnected algebra we allow a finite disjoint union of marked surfaces and apply all constructions componentwise.

\begin{definition}\label{def:arc} \rm
Let $\mathcal{S}$ be a marked surface.
\begin{enumerate}[label=(\arabic*)]
  \item An {\it arc} $\gamma$ in $\mathcal{S}$ is a curve in $\mathcal{S}$ such that its only intersections with $\mathcal{M}$ are the endpoints of $\gamma$.
An arc is called a $\circ$-arc if its endpoints belong to $\mathcal{M}_{\circ}\cup\mathcal{P}_{\circ}$, and an $\rpoint$-arc if both of its endpoints belong to $\mathcal{M}_{\rpoint}\cup\mathcal{P}_{\rpoint}$.
Curves are considered up to homotopy relative to their endpoints.
  \item Two arcs $\gamma_1$ and $\gamma_2$ are {\it intersecting} if they have a common endpoint.
  \item Two arcs $\gamma_1$, $\gamma_2$ are {\it non-crossing} if there exist representatives in their isotopy classes which do not meet anywhere except possibly at endpoints.
The arcs are called {\em crossing} otherwise.
  \item A collection of arcs is {\it non-crossing} if none of the arcs has self-crossings in the interior of $\mathcal{S}$ and the arcs are pairwise non-crossing.
\end{enumerate}
\end{definition}

\begin{definition}\label{def:dissection} \rm
A non-crossing collection \( \mathsf{P} \) of $\rpoint$-arcs is called a \textit{dissection} of \( \mathcal{S} \) if no arc of $\mathsf{P}$ is isotopic to a boundary segment or to a point, and if the arcs cut the surface into polygons each of which contains exactly one $\bpoint$-point from $\mathcal{M}_{\circ}\cup\mathcal{P}_{\circ}$, see Figure \ref{fig:tiles}.
If \( \mathsf{P} \) is a dissection of \( \mathcal{S} \), the pair \( (\mathcal{S}, \mathsf{P}) \) is also called a \defines{tiled surface}.
\end{definition}

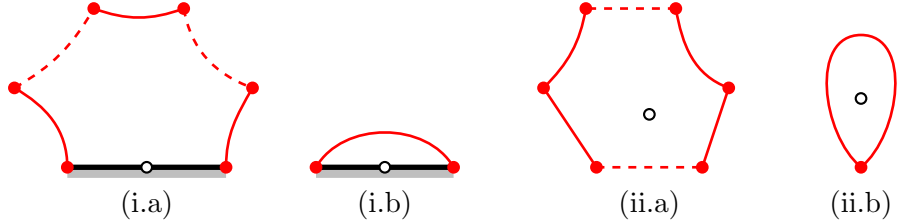
\begin{figure}[htbp]
\centering

\begin{tikzpicture}[>=stealth,scale=0.7]
  \draw [line width=4pt, gray!50] (-6,-0.1) to (-3,-0.1);
  \draw [line width=2pt, black] (-6,0) to (-3,0);
  \draw [line width=4pt, gray!50] (-1.3,-0.1) to (1.3,-0.1);
  \draw [line width=2pt, black] (-1.3,0) to (1.3,0);

  \draw[red, line width=1pt] (-6,0) to[out=90, in=-30](-7,1.5);
  \draw[red, line width=1pt,dashed](-7,1.5)to[out=30, in=-120] (-5.5,3);
  \draw[red, line width=1pt](-5.5,3)to[out=-20,in=-160](-3.8,3);
  \draw[red, line width=1pt,red,dashed](-3.8,3)to[out=-80,in=160](-2.5,1.5);
  \draw[red, line width=1pt] (-2.5,1.5) to[out=-120, in=90](-3,0) ;

  \draw[red, line width=1pt] (-1.3,0) to[out=60, in=120](1.3,0)  ;

  \draw[red, line width=1pt] (6,0) to (6.5,1.5);
  \draw[red, line width=1pt](6.5,1.5) to[out=160, in=-80] (5.5,3);
  \draw[red, line width=1pt,dashed](5.5,3)to(3.8,3);
  \draw[red, line width=1pt](3.8,3)to[out=-100,in=45](3,1.5);
  \draw[red, line width=1pt](3,1.5) to (4,0);
  \draw[red, line width=1pt,dashed](4,0) to(6,0) ;

  \draw[red, line width=1pt,red] (9,0) to[out=140, in=180](9,2.5)  to[out=0, in=40](9,0)  ;

  \draw[thick,black, fill=white ] (-4.5,0) circle (0.1);
  \draw[thick,black, fill=white ] (0,0) circle (0.1);
  \draw[thick,black, fill=white ] (5,1) circle (0.1);
  \draw[thick,black, fill=white ] (9,1.3) circle (0.1);

  \draw[thick,red, fill=red ] (-6,0) circle (0.1);
  \draw[thick,red, fill=red ] (-3,0) circle (0.1);
  \draw[thick,red, fill=red ] (-1.3,0) circle (0.1);
  \draw[thick,red, fill=red ] (1.3,0) circle (0.1);
  \draw[thick,red, fill=red ] (4,0) circle (0.1);
  \draw[thick,red, fill=red ] (6,0) circle (0.1);

  \draw[thick,red, fill=red ] (-7,1.5) circle (0.1);
  \draw[thick,red, fill=red ] (-5.5,3) circle (0.1);
  \draw[thick,red, fill=red ] (-3.8,3) circle (0.1);
  \draw[thick,red, fill=red ] (-2.5,1.5) circle (0.1);
  \draw[thick,red, fill=red ] (5.5,3) circle (0.1);
  \draw[thick,red, fill=red ] (3.8,3) circle (0.1);
  \draw[thick,red, fill=red ] (3,1.5) circle (0.1);
  \draw[thick,red, fill=red ] (6.5,1.5) circle (0.1);

  \draw[thick,red, fill=red ] (9,0) circle (0.1);

  \draw (-4.5,-0.7) node {(i.a)};
  \draw (0,-0.7) node {(i.b)};
  \draw (5,-0.7) node {(ii.a)};
  \draw (9,-0.7) node {(ii.b)};
\end{tikzpicture}

\caption{The types of polygons in a dissection of Definition \ref{def:dissection}.
Figure (i.b) is the special case of a polygon with boundary segments and of size $3$.
Figure (ii.b) is the special case of a polygon of size $1$.}
\label{fig:tiles}

\end{figure}

Let $p \in \mathcal{M}_{\rpoint}\cup\mathcal{P}_{\rpoint}$.
If $p \in \mathcal{M}_{\rpoint}$, let $p'$ and $p''$ be two points in the same boundary component such that $p', p'' \not\in \mathcal{M}$ and $p$ is the only marked point in $\mathcal{M}$ lying in the boundary segment between $p'$ and $p''$.
Consider a curve $\delta$ isotopic to this boundary segment but such that the only intersection with the boundary of $\mathcal{S}$ is at its endpoints.
If $p \in \mathcal{P}_{\rpoint}$ , consider a closed simple curve $\delta$ in the interior of $\mathcal{S}$ around $p$ and with no intersections with $\mathcal{M}$.

At a boundary point $p$, the \defines{complete fan at $p$} is the finite sequence of arc incidences crossed by $\delta$ in clockwise order.
At a puncture it is the bi-infinite periodic sequence obtained by repeating this cyclic order.
A \defines{fan} is a finite interval of consecutive incidences.
Thus a fan at a puncture may wind several times around it; angles are identified modulo one period.
If an arc has both ends at $p$, its two incidences are distinguished.

Any two consecutive crossings of the curve $\delta$ with arcs in $\mathsf{P}$ define a triangle whose vertices are these two crossings and the marked point $p$; see Figure~\ref{fig:angle}.
This triangle is called an {\it angle at $p$} (also called {\it angle of $\mathsf{P} $}), and we say the curve $\delta$ {\it cuts this angle at $p$}.

\begin{figure}[ht!]
\centering
\begin{tikzpicture}[scale=1.2]
\fill[gray!50] (-1.5,0) rectangle (1.5,0.1);
\draw[line width=1.5pt] (-1.5,0) -- (1.5,0);
\node at (-1,0.25) {$p'$};
\node at (1,0.25) {$p''$};

\draw[green, fill=green!10]
(0,0) -- (0.6,-0.6)
arc(-45:-90:0.85)
-- cycle;

\node[right] at (0.7,-0.5) {$\delta$};
\draw[red] (0,0) -- (1,-1) ;
\draw[red] (0,0)--(-1,-1) ;
\draw[red] (0,0)--(0,-1.35) ;

\draw[black!80] (0,0) -- (0.85,0) arc(0:-180:0.85) -- cycle;
\fill[red] (0,0) circle (2pt) node[above] {$p$};
\end{tikzpicture}
\hspace{3em}
\begin{tikzpicture}[scale=1.2]

\node[right] at (0.6,-0.8) {$\delta$};
\draw[green, fill=green!10] (0,0) -- (1,0) arc(0:-60:1) -- cycle;
\draw[red] (0,0) -- (1.25,0) ;
\draw[red] (0,0) -- (0.6,-1) ;
\draw[red] (0,0) -- (-0.8,1) ;
\draw[red] (0,0) -- (-0.2,-1.1);
\fill[red] (0,0) circle (2pt) node[above right] {$p$};

\draw [black!80](0,0) circle (1);
\end{tikzpicture}
\hspace{3em}
\begin{tikzpicture}[scale=1.2]
\draw[black!80, fill=green!10] (0,0) circle (1);
\fill[red] (0,0) circle (2pt) node[above] {$p$};
\draw[red] (0,0) -- (0,-1.2) ;
\node[right] at (0.7,-0.7) {$\delta$};

\end{tikzpicture}

\caption{The shaded triangle is an angle at $p$.
A puncture $p \in \mathcal{P}_{\rpoint}$ can have a unique angle which is a self-folded triangle.}
\label{fig:angle}
\end{figure}
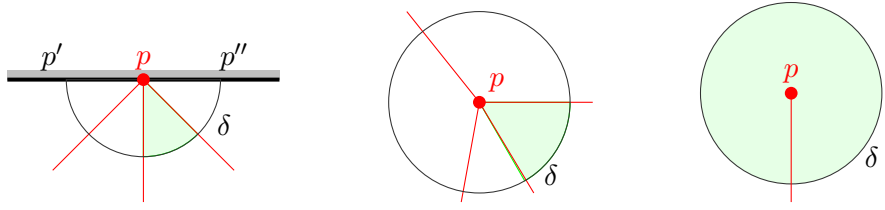

\begin{definition} \rm
Let \( (\mathcal{S}, \mathsf{P}) \) be a tiled surface.
We define \( Q := Q(\mathcal{S}, \mathsf{P}) \) to be the quiver defined as follows:
\begin{enumerate}[label=(\arabic*)]
  \item The vertices in $Q$ are in bijection with the arcs in $\mathsf{P}$.
  \item The arrows in $Q$ are in bijection with the angles at the marked points in $ \mathcal{M}_{\rpoint}\cup\mathcal{P}_{\rpoint}$.
In other words, there is an arrow $a\colon x_i \to x_j$ in $Q$ if and only if the corresponding arcs ${x}_i$ and ${x}_j$ share an endpoint $p_a \in \mathcal{M}_{\rpoint}\cup\mathcal{P}_{\rpoint}$ and ${x}_j$ is the immediate successor of ${x}_i$ in the complete fan at $p_a$.
\end{enumerate}
\end{definition}

\begin{definition}\label{def:label} \rm
Let \( (\mathcal{S}, \mathsf{P}) \) be a tiled surface.
\begin{enumerate}[label=(\arabic*)]
  \item A {\it label} at a point $p\in \mathcal{M}_{\rpoint}\cup\mathcal{P}_{\rpoint}$ is a finite fan at $p$ of length $\geqslant 3$.
  \item Given two labels $\ell$ and $\ell'$, if $\ell$ is a consecutive subsequence of $\ell'$, then we say the label $\ell'$ is {\it redundant}.
\end{enumerate}
\end{definition}

\begin{notation} \rm
We will draw a label $\ell$ as a blue curve $\delta$ around the corresponding marked point with an orientation such that $\delta (0)$ is a point in the first arc of $\ell$, $\delta (1)$ is a point in the last arc of $\ell$ and if $\ell$ contains an arc ${x}$ multiple times, the corresponding intersections of $\delta$ with arc ${x}$ are at distinct points.
This is in order to make it clear what the label is from $\delta$, especially in the case when the label starts and ends at the same arc.
See Example~\ref{ex:not unique} and Figure~\ref{fig:non-homotopic}.
\end{notation}

\begin{definition}\label{def:tiled surface} \rm
Let \( (\mathcal{S}, \mathsf{P}) \) be a tiled surface and \( \mathsf{L} \) be a finite set of labels.
Let \( Q = Q(\mathcal{S}, \mathsf{P}) \) and consider the path algebra \( \mathbf{k}Q \).
We define \( I_{\mathsf{P},\mathsf{L}} \) to be the ideal of \( \mathbf{k}Q \) generated by the following paths (see Figure \ref{fig:relation}):
\begin{enumerate}[label=(R\arabic*)]
\item $ab$ of length 2 such that $p_{a} \neq p_{b}$ or $p_{a} = p_{b}$ and $t(a) = s(b)$ corresponds to a loop arc.
\label{R1}

\item linearly oriented paths in \( Q \) corresponding to a label in \( \mathsf{L} \).
\label{R2}
\end{enumerate}

Relations of type \ref{R1} come from the tiles in \( (\mathcal{S}, \mathsf{P}) \), as certain compositions of two arrows within a polygon, and these are exactly the relations used in the construction of (locally) gentle algebras via surfaces.
For string algebras, we need also to consider labels and corresponding relations of type \ref{R2}.
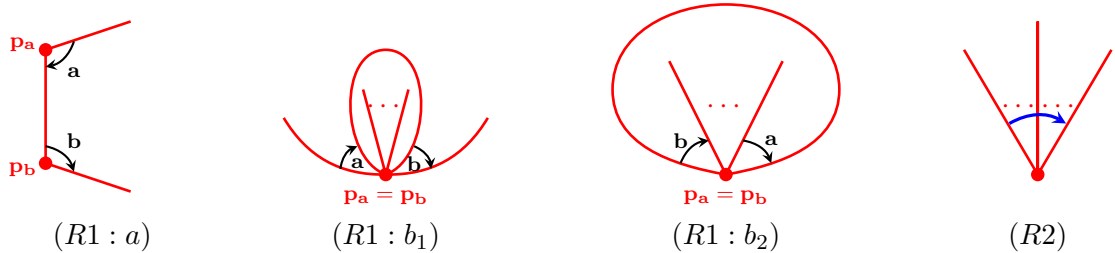
\begin{figure}[ht]
\centering
\begin{tikzpicture}[>=stealth,scale=0.75]
	% ========== ?1?? (R1:a) ==========
	\begin{scope}
		\draw [line width=1pt, red] (-6.5,-0.5) to(-8,0)to(-8,2)to(-6.5,2.5);
		\draw [bend left, line width=0.8pt, black,->] (-7.5,2.15) to (-8,1.7);
		\draw [bend left, line width=0.8pt, black,->] (-8,0.3) to (-7.5,-0.15);
		\draw[thick,red, fill=red ] (-8,0) circle (0.1);
		\draw[thick,red, fill=red ] (-8,2) circle (0.1);
		\draw (-7.5,1.6) node[font=\scriptsize] {$\mathbf{a}$};
		\draw (-7.5,0.4) node[font=\scriptsize] {$\mathbf{b}$};
		\draw (-8.4,2.1) node[red,font=\scriptsize] {$\mathbf{p_a}$};
		\draw (-8.4,-0.1) node[red,font=\scriptsize] {$\mathbf{p_b}$};
		\draw (-7,-1.3) node {$(R1:a)$};
	\end{scope}
	
	% ========== ?2?? (R1:b1)?xshift=+1cm ==========
	\begin{scope}[xshift=1.5cm]
		\draw [line width=1pt, red] (-3.9,1.3) to(-3.5,-0.2)to(-3.1,1.3);
		\draw[bend right, line width=1pt,red] (-5.3,0.8) to(-3.5,-0.2);
		\draw[bend right, line width=1pt,red] (-3.5,-0.2) to(-1.7,0.8);
		\draw[red, line width=1pt] (-3.5,-0.2) to[out=160, in=180](-3.5,2)to[out=0, in=20](-3.5,-0.2);
		\draw [bend left, line width=0.8pt, black,->] (-4.3,-0.1) to (-4,0.3);
		\draw [bend left, line width=0.8pt, black,->] (-3,0.3) to (-2.7,-0.1);
		\draw[thick,red, fill=red ] (-3.5,-0.2) circle (0.1);
		\draw (-3.5,1) node[red] {$\cdots$};
		\draw (-4,0) node[font=\scriptsize] {$\mathbf{a}$};
		\draw (-3,0) node[font=\scriptsize] {$\mathbf{b}$};
		\draw (-3.5,-0.6) node[red,font=\scriptsize] {$\mathbf{p_{a}=p_{b}}$};
		\draw (-3.5,-1.3) node {$(R1:b_{1})$};
	\end{scope}
	
	% ========== ?3?? (R1:b2)?xshift=+2cm ==========
	\begin{scope}[xshift=3cm]
		\draw [line width=1pt, red] (0,1.8) to (1,-0.2)to (2,1.8);
		\draw[red, line width=1pt] (1,-0.2) to[out=170, in=-90](-1,1.3)to[out=90, in=180](1,2.8)to[out=0, in=90](3,1.3)to[out=-90, in=10](1,-0.2);
		\draw [bend left, line width=0.8pt, black,->] (0.2,0) to (0.7,0.4);
		\draw [bend left, line width=0.8pt, black,->] (1.3,0.4) to (1.8,0);
		\draw[thick,red, fill=red ] (1,-0.2) circle (0.1);
		\draw (1,1) node[red] {$\cdots$};
		\draw (0.2,0.4) node[font=\scriptsize] {$\mathbf{b}$};
		\draw (1.8,0.4) node[font=\scriptsize] {$\mathbf{a}$};
		\draw (1,-0.6) node[red,font=\scriptsize] {$\mathbf{p_{a}=p_{b}}$};
		\draw (1,-1.3) node {$(R1:b_{2})$};
	\end{scope}
	
	% ========== ?4?? (R2)?xshift=+3cm ==========
	\begin{scope}[xshift=5cm]
		\draw [line width=1pt, red] (3.2,2)to (4.5,-0.2) to (4.5,2.5)to (4.5,-0.2)to (5.8,2);
		\draw [bend left, line width=1.2pt, blue,->] (4,0.7) to (5,0.7);
		\draw[thick,red, fill=red ] (4.5,-0.2) circle (0.1);
		\draw (4.2,1) node[red] {$\cdots$};
		\draw (4.9,1) node[red] {$\cdots$};
		\draw (4.5,-1.3) node {$(R2)$};
	\end{scope}
\end{tikzpicture}

\caption{This figure illustrates the ideal \( I_{\mathsf{P},\mathsf{L}} \) in $\mathbf{k}Q$ of the tiled surface $(\mathcal{S},\mathsf{P})$ from Definition \ref{def:tiled surface}.}
\label{fig:relation}

\end{figure}
\end{definition}

\begin{definition}
Let \( (\mathcal{S}, \mathsf{P}) \) be a tiled surface and let \( Q = Q(\mathcal{S}, \mathsf{P}) \).
Let \( \mathsf{L} \) be a finite set of labels on \( \mathcal{S} \) such that for every \( q \in \mathcal{P}_{\rpoint} \), \( \mathsf{L} \) contains a label at \( q \).
The \textit{labelled tiling algebra} \( A_{\mathsf{P},\mathsf{L}} \) associated to the data \( (\mathcal{S}, \mathsf{P}, \mathsf{L}) \) is the bound quiver algebra \( A_{\mathsf{P},\mathsf{L}} = \mathbf{k}Q/I_{\mathsf{P},\mathsf{L}} \).
\end{definition}

If a label contains another label as a consecutive subsequence, its relation already belongs to the ideal generated by the shorter one.
We may therefore, and shall, assume that $\mathsf{L}$ contains no redundant labels.

The following result shows that the class of labelled tiling algebras coincides with the class of string algebras.

\begin{theorem}[{\!\!\cite[Theorem~3.1]{BC24}}]\label{thm:string-labelled-tiling}
Let \( A = \mathbf{k}Q/I \) be a finite-dimensional monomial algebra.
The following are equivalent.
\begin{enumerate}[label={\rm(\arabic*)}]
  \item \( A \) is a string algebra.
  \item \( A = (\mathbf{k}Q/J)/(I/J) \), where \( \mathbf{k}Q/J \) is a locally gentle algebra.
  \item \( A \) is a labelled tiling algebra of a marked surface.
\end{enumerate}
\end{theorem}

By Theorem~\ref{thm:string-labelled-tiling}, the labelled tiled surface associated with a string algebra is not unique in general: different choices of the locally gentle algebra $\mathbf{k}Q/J$ may produce non-homeomorphic surfaces.
We fix one labelled tiled surface $(\mathcal{S},\mathsf{P},\mathsf{L})$ associated with $A$ for the remainder of the paper.
All algebraic statements are independent of this choice, although the representatives of the corresponding curves depend on it.
Example~\ref{ex:not unique} illustrates this phenomenon.
The uniqueness problem for such geometric models was studied in
\cite{XZ26}, where the authors characterize the string algebras whose
labelled tiled-surface models are unique up to equivalence.

\begin{example}\rm \label{ex:not unique}
Consider the quiver $Q$ and the admissible ideal $I_{1}=\langle bca,ad,ab\rangle$ in Example \ref{ex:string-gentle algebra}.
There are two choices of associated locally gentle algebras: the finite-dimensional algebra $B_1 = \mathbf{k}Q/J_1$ with $J_1 = \langle ab\rangle$ and the infinite-dimensional algebra $B_2 = \mathbf{k}Q/J_2$ with $J_2 = \langle ad\rangle$.
The corresponding surfaces, shown in Figure \ref{fig:non-homotopic}, are not homeomorphic.
\begin{figure}[htbp]
\begin{center}
\begin{tikzpicture}[scale=0.35]
  \begin{scope}[shift={(-12,0)}]
  \draw[line width=1.5pt,fill=white] (0,0) circle (6cm);
  \draw[line width=1.5pt,fill=gray!50] (0,0) circle (1cm);

  \path
  (0:1) coordinate (b1)
  (-90:1) coordinate (b2)
  (-180:1) coordinate (b3)
  (-90:6) coordinate (b4)
  (90:6) coordinate (b5)
  (0:6) coordinate (r1)
  (180:6) coordinate (r2)
  (90:1) coordinate (r3);

  \draw[bend left,line width=1pt,red] (b4) to (b3);
  \draw[bend right,line width=1pt, red] (b4) to (b1);

  \draw[red,line width=1pt,red] (b4) to[out=30,in=-90](3.5,0)to[out=90,in=0](0,4)to[out=180,in=90](-3.5,0) to[out=-90,in=150](b4);
  \draw[red,line width=1pt,red] (b4) to[out=10,in=-90](5.2,0)to[out=90,in=-10](b5);

  \draw[bend left,line width=1.2pt,blue,->] (-2.2,-4) to (2.2,-4);
  \draw[bend left,line width=1.2pt,blue,->] (1.4,-3) to (4,-3.4);

  \draw[thick, red ,fill=red]
  (b1) circle (0.2cm)
  (b3) circle (0.2cm)
  (b4) circle (0.2cm)
  (b5) circle (0.2cm);

  \draw[thick,black, fill=white]
  (r1) circle (0.2cm)
  (b2) circle (0.2cm)
  (r2) circle (0.2cm)
  (r3) circle (0.2cm);

  \draw(-1.8,-1) node[red,font=\scriptsize] {$\mathbf{x_{3}}$};
  \draw(2,-2) node[red,font=\scriptsize] {$\mathbf{x_{1}}$};
  \draw(0,3) node[red,font=\scriptsize] {$\mathbf{x_{2}}$};
  \draw(4.7,0) node[red,font=\scriptsize] {$\mathbf{x_{4}}$};

  \end{scope}

  \begin{scope}[shift={(8,0)}]
  \draw[line width=1.5pt,fill=white] (0,0) circle (6cm);

  \path
  (0:6) coordinate (b1)
  (-90:6) coordinate (b2)
  (180:6) coordinate (b3)
  (90:6) coordinate (b4)
  (2,0) coordinate (b6)
  (45:6) coordinate (r1)
  (135:6) coordinate (r2)
  (-135:6) coordinate (r3)
  (-45:6) coordinate (r4);

  \draw[line width=1pt,red] (b6) to (b1);
  \draw[line width=1pt,red] (b6) to (b2);
  \draw[line width=1pt,red] (b6) to (b4);
  \draw[line width=1pt,red] (b4) to (b3);

% \draw[red,line width=1.2pt,blue,->] (1.8,0.5) to[out=10,in=90](2.8,0)to[out=-90,in=0](1.8,-0.7)to[out=180,in=-150](1.8,0.9);
  \draw[blue,line width=1.2pt,->] (1.8,0.5) arc(100:0:0.6) arc(0:-275:0.8);
% \draw[red,line width=1.2pt,blue,->] (1.6,-1.45) to[out=140,in=180](1.3,1.7)to[out=0,in=90](3.3,0);
  \draw[blue,line width=1.2pt,->] (1.6,-1.45) arc(255:0:1.6);

  \draw[thick,red ,fill=red]
  (b1) circle (0.2cm)
  (b2) circle (0.2cm)
  (b3) circle (0.2cm)
  (b4) circle (0.2cm)
  (b6) circle (0.2cm);

  \draw[thick,black, fill=white]
  (r1) circle (0.2cm)
  (r2) circle (0.2cm)
  (r3) circle (0.2cm)
  (r4) circle (0.2cm);

  \draw(0.25,-3.5) node[red,font=\scriptsize] {$\mathbf{x_{1}}$};
  \draw(0.45,3) node[red,font=\scriptsize] {$\mathbf{x_{2}}$};
  \draw(4.8,0.5) node[red,font=\scriptsize] {$\mathbf{x_{3}}$};
  \draw(-2,3) node[red,font=\scriptsize] {$\mathbf{x_{4}}$};
  \end{scope}

  \end{tikzpicture}
  \end{center}
\caption{Non-homeomorphic surfaces corresponding to the same string algebra $\mathbf{k}Q/I_{1}$ in Example \ref{ex:string-gentle algebra}.}
\label{fig:non-homotopic}
\end{figure}
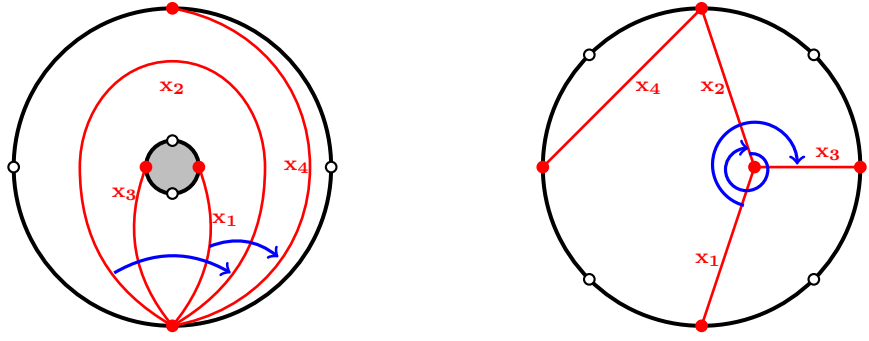
\end{example}

\subsection{Permissible arcs and good crossings}\label{subsec:permissible-good-crossings}

We now translate the module category of the string algebra $A = A_{\mathsf{P},\mathsf{L}}$ into the geometric setting of the labelled tiled surface $(\mathcal{S}, \mathsf{P}, \mathsf{L})$, following the approach established in \cite{BC24}.

Recall that the vertices of $Q(\mathcal{S}, \mathsf{P})$ correspond to the $\rpoint$-arcs in the dissection $\mathsf{P}$.
The indecomposable string modules are represented by certain $\circ$-arcs.
Band modules require certain closed curves together with the rank and automorphism parameters recalled above.
All curves in this subsection are transverse to $\mathsf{P}$ and meet it finitely many times in minimal position.
A curve with no crossing with $\mathsf{P}$ is called a \defines{zero arc} and represents the zero module.
Module arcs used to parametrize indecomposable modules are nonzero.

Suppose that a curve crosses $\mathbf{x}_i,\mathbf{x}_j\in\mathsf{P}$ consecutively.
It \defines{cuts an angle} if the segment between these crossings and the relevant incidences of $\mathbf{x}_i,\mathbf{x}_j$ bound the corresponding angle at their common red endpoint.

\begin{definition}[{\!\!\cite[Definition~4.3]{BC24}}]\label{def:permissible} \rm
A curve is \defines{permissible} if each pair of consecutive crossings with $\mathsf{P}$ cuts an angle and no segment follows all the successive angles of a label in $\mathsf{L}$.
A permissible $\circ$-arc or $\rpoint$-arc is a permissible curve with the respective endpoints of Definition~\ref{def:arc}.
A closed curve is permissible if the angle condition holds cyclically and no segment of any positive power follows a label.
\end{definition}

An oriented nonzero permissible arc crossing $\mathbf{x}_{i_0},\ldots,\mathbf{x}_{i_d}$ reads a string $w_\gamma$ of length $d$ from the intervening angles and their directions.
When $d=0$, it reads the signed trivial string at $i_0$, with sign determined by the orientation and the chosen sign functions as in \cite[Theorem~4.6]{BC24}.
In particular, this one-crossing arc represents the simple module at $i_0$, whereas a zero arc represents zero.
We set $M_\gamma=M(w_\gamma)$.

Two oriented nonzero permissible arcs are \defines{equivalent} if they have the same crossed-arc sequence and cut the same intervening oriented angles; for a one-crossing arc the signed trivial strings must agree.
The end segments do not otherwise enter this equivalence.
For unoriented permissible $\circ$-arcs corresponding to string modules, an arc is identified with its reversal, corresponding to the equivalence $w\sim w^{-1}$.
Thus one-crossing arcs through different dissection arcs are not equivalent.
Every such class has a unique permissible $\circ$-arc representative up to isotopy.

\begin{theorem}[{\!\!\cite[Theorem~4.6]{BC24}}]
\label{thm:geometric_modules}
  Let $A=A_{\mathsf{P},\mathsf{L}}$.
  The preceding conventions {provide} bijections between:
  \begin{enumerate}[label={\rm(\arabic*)}]
  \item equivalence classes of oriented nonzero permissible arcs and the strings of $A$; \item classes of permissible closed curves meeting $\mathsf{P}$ at least twice and positive powers of bands.
  \end{enumerate}
\end{theorem}

We recall the description of morphisms by string combinatorics, following \cite{BC24}.
Let $w$ be a string.
A \defines{factor string decomposition} of $w$ is a decomposition of the form $w=w_\ell w_e w_r$, such that either $w_\ell$ is a trivial string or the last letter of $w_\ell$ lies in $Q_1^{-1}$ and either $w_r$ is trivial or the first letter of $w_r$ lies in $Q_1$.
In this case, $w_e$ is called a \defines{factor substring} of $w$.
Similarly, a decomposition of $w$ of the form $w=w_\ell w_e w_r$, where either $w_\ell$ is a trivial string or the last letter of $w_\ell$ lies in $Q_1$ and either $w_r$ is trivial or the first letter of $w_r$ lies in $Q_1^{-1}$, is called a \defines{substring decomposition}.
In this case, $w_e$ is called an \defines{image substring} of $w$.

The set of equivalence classes of factor string (resp. substring) decompositions of $w$ is denoted by $\Fac (w)$ (resp. $\Sub(w)$).
Let $v, w$ be two strings.
We call a pair $(v_\ell v_e v_r, w_\ell w_e w_r)$ in $\Fac (v) \times \Sub (w)$ an \defines{admissible pair} if $w_e = v_e$ or $w_e = v_e^{-1}$.
We will also sometimes call the string $w_e$ an \defines{overlap between $w$ and $v$}.

Homomorphisms between string modules arise from admissible pairs as has been shown by Crawley-Boevey:

\begin{theorem}[\!\!\cite{CB89}]\label{thm:homs}
For strings $v,w$, the maps obtained by projecting $M(v)$ onto a factor substring, identifying the overlap (reversing its basis order if necessary), and including it as an image substring of $M(w)$ form a $\mathbf{k}$-basis of $\Hom_A(M(v),M(w))$.
The basis is indexed by the admissible pairs in $\Fac(v)\times\Sub(w)$.
\end{theorem}

We recall a combinatorial description of morphisms between string modules over a labelled tiling algebra, using permissible $\circ$-arcs.
Such a description has been given in \cite{C26} in the case when the algebra is gentle, and the version for string algebras is provided in \cite{BC24}.

\begin{definition}\rm
  Let $\gamma$ and $\beta$ be two permissible $\circ$-arcs, and suppose there is a crossing $c$ between $\gamma$ and $\beta$ in the interior of $\mathcal{S}$.
  Then there must be an arc $\mathbf{x}$ in $\mathsf{P}$ crossed by both arcs.
  The arcs $\mathbf{x}, \gamma$ and $\beta$ enclose a simply connected triangle $\delta$ with $c$ as one of its vertices.
  We say the crossing $c$ is a \defines{crossing from $\beta$ to $\gamma$} if $\gamma$ follows $\beta$ in the clockwise order around $\delta$, see Figure~\ref{fig:crossing}.
\end{definition}

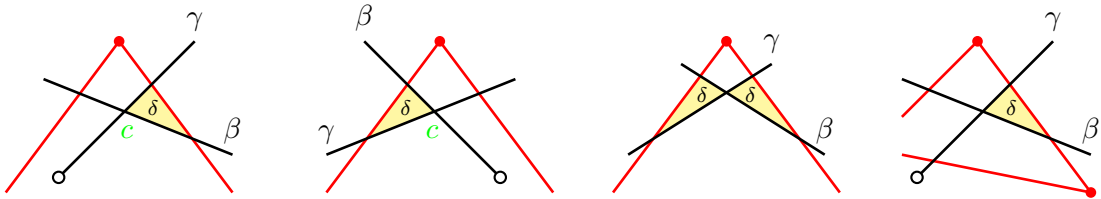
\begin{figure}[ht]
\centering
\begin{tikzpicture}[scale=1]
\fill[yellow!45] (0.42,-0.57) to(0.96,-1.29) to (0.07,-0.93) -- cycle;

\draw[line width=1pt,red] (0,0)--(-1.5,-2) ;
\draw[line width=1pt,red] (0,0) -- (1.5,-2) ;

\draw[black, line width=1pt] (-0.8,-1.8) -- (1,0)node[above] {$\gamma$} ;
\draw[black, line width=1pt] (-1,-0.5) -- (1.5,-1.5) node[above] {$\beta$};

\draw(0.1,-1.2) node[green] {$c$};
\draw(0.45,-0.88) node[black] {\scriptsize$\delta$};
\fill[red] (0,0) circle (2pt) ;
\fill[white] (0,-2) circle (2pt) ;
\draw[thick,black, fill=white] (-0.8,-1.8) circle (0.075cm);
\end{tikzpicture}
\hspace{0.5cm}
\begin{tikzpicture}[scale=1]
\fill[yellow!45] (-0.42,-0.57) to(-0.96,-1.29) to (-0.07,-0.93) -- cycle;

\draw[line width=1pt,red] (0,0)--(1.5,-2) ;
\draw[line width=1pt,red] (0,0) -- (-1.5,-2) ;

\draw[black, line width=1pt] (0.8,-1.8) -- (-1,0)node[above] {$\beta$} ;
\draw[black, line width=1pt] (1,-0.5) -- (-1.5,-1.5) node[above] {$\gamma$};

\draw(-0.1,-1.2) node[green] {$c$};
\draw(-0.45,-0.88) node[black] {\scriptsize$\delta$};
\fill[red] (0,0) circle (2pt) ;
\fill[white] (0,-2) circle (2pt) ;
\draw[thick,black, fill=white] (0.8,-1.8) circle (0.075cm);
\end{tikzpicture}
\hspace{0.5cm}
\begin{tikzpicture}[scale=1]
\fill[yellow!45] (-0.9675,-1.29) to(-0.3455,-0.4607) to (0,-0.6789) -- cycle;
\fill[yellow!45] (0.9675,-1.29) to(0.3455,-0.4607) to (0,-0.6789) -- cycle;

\draw[line width=1pt,red] (0,0)--(-1.5,-2) ;
\draw[line width=1pt,red] (0,0) -- (1.5,-2) ;
\draw[black, line width=1pt] (-1.3,-1.5) -- (0.6,-0.3)node[above] {$\gamma$} ;
\draw[black, line width=1pt] (-0.6,-0.3) -- (1.3,-1.5) node[above] {$\beta$};

\draw(-0.32,-0.7) node[black] {\scriptsize$\delta$};
\draw(0.32,-0.7) node[black] {\scriptsize$\delta$};
\fill[red] (0,0) circle (2pt) ;
\fill[white] (0,-2) circle (2pt) ;

\end{tikzpicture}
\hspace{0.5cm}
\begin{tikzpicture}[scale=1]
\fill[yellow!45] (0.42,-0.57) to(0.96,-1.29) to (0.07,-0.93) -- cycle;

\draw[line width=1pt,red] (-1,-1.5) -- (1.5,-2) ;
\draw[line width=1pt,red] (0,0)--(-1,-1) ;
\draw[line width=1pt,red] (0,0) -- (1.5,-2) ;

\draw[black, line width=1pt] (-0.8,-1.8) -- (1,0)node[above] {$\gamma$} ;
\draw[black, line width=1pt] (-1,-0.5) -- (1.5,-1.5) node[above] {$\beta$};

\draw(0.45,-0.88) node[black] {\scriptsize$\delta$};
\fill[red] (0,0) circle (2pt) ;
\fill[red] (1.5,-2) circle (2pt) ;
\fill[white] (0,-2) circle (2pt);
\draw[thick,black, fill=white] (-0.8,-1.8) circle (0.075cm);
\end{tikzpicture}

\caption{A crossing between $\circ$-arcs, from $\beta$ to $\gamma$.}
\label{fig:crossing}
\end{figure}

A given crossing is incident with at most two opposite triangles, both of which {provide} the same orientation.
This orientation is independent of the choice of representatives in their isotopy classes.

\begin{definition}\rm
 Let $\gamma$ and $\beta$ be two permissible $\circ$-arcs.
\begin{enumerate}[label=(\arabic*)]
  \item We say there is an \defines{intersection from $\beta$ to $\gamma$} if there are end-segments $\beta_0$ and $\gamma_0$ of $\beta$ and $\gamma$, respectively, sharing an endpoint $x \in \mathcal{M}_{\circ}\cup\mathcal{P}_{\circ}$, such that $\gamma_0$ follows $\beta_0$ in the anticlockwise direction around $x$.
  \item A \defines{trivial intersection from $\beta$ to $\beta$} is an intersection in which the two end-segments coincide.
  \item An intersection from $\beta$ to $\gamma$ is \defines{traversed by $\mathsf{P}$} if there is an arc $\tau$ in $\mathsf{P}$ which crosses both $\beta$ and $\gamma$ and such that $\beta$, $\gamma$, and $\tau$ enclose a simply connected triangle, see Figure~\ref{fig:intersection}.
\end{enumerate}
\end{definition}

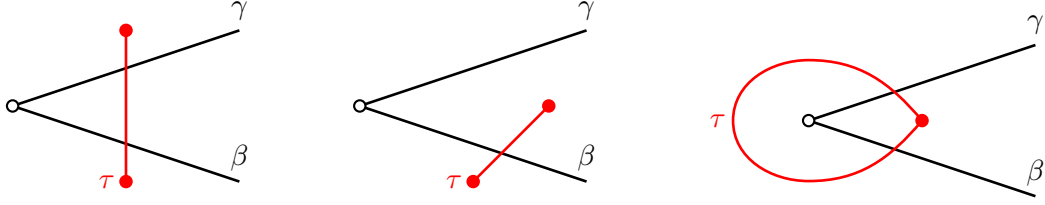
\begin{figure}[htbp]
\centering
\begin{tikzpicture}[scale=1]
\draw[line width=1pt] (0,0)--(3,1)node[above] {$\gamma$} ;
\draw[line width=1pt] (0,0) -- (3,-1) node[above] {$\beta$};
\draw[red,line width=1pt] (1.5,1) -- (1.5,-1) node[left] {$\tau$};

\draw[thick,black, fill=white] (0,0) circle (0.075cm);
\draw[thick,red, fill=red] (1.5,-1) circle (0.075cm);
\draw[thick,red, fill=red] (1.5,1) circle (0.075cm);
\end{tikzpicture}
\hspace{1cm}
\begin{tikzpicture}[scale=1]
\draw[line width=1pt] (0,0)--(3,1)node[above] {$\gamma$} ;
\draw[line width=1pt] (0,0) -- (3,-1) node[above] {$\beta$};
\draw[red,bend left,line width=1pt] (2.5,0) -- (1.5,-1) node[left] {$\tau$};

\draw[thick,black, fill=white] (0,0) circle (0.075cm);
\draw[thick,red, fill=red] (1.5,-1) circle (0.075cm);
\draw[thick,red, fill=red] (2.5,0) circle (0.075cm);
\end{tikzpicture}
\hspace{1cm}
\begin{tikzpicture}[scale=1]
\draw[line width=1pt] (0,0)--(3,1)node[above] {$\gamma$} ;
\draw[line width=1pt] (0,0) -- (3,-1) node[above] {$\beta$};
\draw(-1.2,0) node[red] {$\tau$};

\draw[line width=1pt,red] (1.5,0) to[out=130, in=0] (0,0.8) to[out=180,in=90](-1,0)to[out=-90,in=180](0,-0.8) to[out=0,in=-130](1.5,0) ;

\draw[thick,black, fill=white] (0,0) circle (0.075cm);
\draw[thick,red, fill=red] (1.5,0) circle (0.075cm);
\end{tikzpicture}

\caption{Left: Intersection from $\beta$ to $\gamma$ traversed by $\mathsf{P}$.
Middle: Intersection from $\beta$ to $\gamma$ not traversed by $\mathsf{P}$.
Right: Intersection from $\beta$ to $\gamma$ not traversed by $\mathsf{P}$ and intersection from $\gamma$ to $\beta$ traversed by $\mathsf{P}$.}
\label{fig:intersection}
\end{figure}

Denote by $\mathsf{C}(\beta,\gamma)$ the number of crossings from $\beta$ to $\gamma$, and by $\mathsf{I}_{\circ,\mathsf{P}}(\beta,\gamma)$ the number of intersections from $\beta$ to $\gamma$ traversed by $\mathsf{P}$.
Set $\delta_{\beta,\gamma}=1$ if $\beta$ and $\gamma$ are isotopic and $\delta_{\beta,\gamma}=0$ otherwise.

\begin{theorem}[{\!\!\cite[Theorem~4.11]{BC24}}]\label{thm:geometric-hom}
Let $\beta$ and $\gamma$ be permissible $\circ$-arcs, and let $M_\beta$ and $M_\gamma$ be the corresponding string modules.
Then
\[
\dim_{\mathbf{k}} \Hom_A(M_\beta, M_\gamma) = \mathsf{C}(\beta,\gamma) + \mathsf{I}_{\circ,\mathsf{P}}(\beta,\gamma) - \delta_{\beta,\gamma}.
\]
\end{theorem}

For distinct permissible $\circ$-arcs corresponding to string modules, the graph maps are indexed by their oriented crossings and intersections with $\mathsf P$.
On the diagonal there are two trivial endpoint intersections, both representing the same identity map.
We retain this identity only once; this accounts for the term $-\delta_{\beta,\gamma}$ in the formula.
We call the resulting basis the \defines{geometric graph-map basis}.
Its maps are defined by the admissible overlaps in Theorem~\ref{thm:homs}, not by the dimension formula alone.

The $\circ$-arc representative is convenient for morphisms.
For $\tau$-rigidity we use the support arcs of \cite[Definitions~4.2 and~4.14]{BC24}.
Their endpoints may be red marked points, or boundary white points in a tile having a single dissection side, as in type~$(i.b)$ of Figure~\ref{fig:tiles}.
Fix the interior segments determined by the string.
In each end tile choose an admissible end segment as far clockwise as its equivalence class permits: adjoining the incident tile side in that direction must change the class.
Apply this rule separately at the two endpoints using their local clockwise orders.
This selects the clockwise-most representative up to isotopy, denoted by $\gamma_M^{\circlearrowright}$ for a string module $M$. Figure~\ref{fig:equivalent-arc-representatives}
illustrates a permissible $\circ$-arc $\gamma_M^\circ$
and its equivalent clockwise-most representative
$\gamma_M^{\circlearrowright}$.
The two arcs cross the same dissection arcs in the same order
and cut the same intervening angles, hence represent the same
string module $M$.

\begin{figure}[htbp]
	\centering
	\begin{tikzpicture}[
		xscale=1.10,
		yscale=1.04,
		>=stealth,
		dissection/.style={red,thick},
		whitepoint/.style={
			circle,draw=black,fill=white,
			inner sep=0pt,minimum size=4pt
		}
		]
		
		%------------------------------------------------
		% 1. ?????
		%------------------------------------------------
		\coordinate (A) at (0,0.85);
		\coordinate (B) at (0.45,2.05);
		\coordinate (C) at (2,2.05);
		\coordinate (D) at (2,-0.45);
		\coordinate (E) at (0.4,-0.3);
		
		\coordinate (F) at (3.2,-0.45);
		\coordinate (G) at (4.3,2.05);
		\coordinate (H) at (5.4,-0.45);
		\coordinate (U) at (6.35,2.25);
		\coordinate (V) at (7.45,-0.45);
		
		% J?K ???????????????
		\coordinate (I) at (7.9,1.9);
		\coordinate (J) at (9.5,1.7);
		\coordinate (K) at (9.5,-0.1);
		
		% ????????
		\coordinate (Wl) at (0.95,0.95);
		\coordinate (Wr) at (9.5,0.7);
		
		%------------------------------------------------
		% 2. ????
		%------------------------------------------------
		\draw[dissection] (C)--(B)--(A)--(E)--(D);
		
		% ????????????????
		\draw[dissection] (I)--(J);
		\draw[gray!50,line width=4pt] (9.55,2)--(9.55,-0.5);
		
		\draw[black,very thick] (9.5,2)--(9.5,-0.5);
		
		\draw[dissection] (K)--(V);
		
		\draw[dissection] (C)--(D);
		\draw[dissection] (C)--(F);
		
		\draw[dissection]
		(F)
		.. controls (3.45,0.30) and (4.12,1.28) ..
		(G);
		
		\draw[dissection] (G)--(H);
		
		\draw[dissection]
		(H)
		.. controls (5.52,0.35) and (6.18,1.50) ..
		(U);
		
		\draw[dissection] (U)--(V);
		\draw[dissection] (V)--(I);

		%------------------------------------------------
		% 4. ??-arc???
		%------------------------------------------------
		\draw[blue,very thick]
		(Wl)
		.. controls (2.7,0.55) and (7.0,0.50) ..
		(Wr);
		
		\node[blue] at (4.2,0.3)
		{$\gamma_M^{\circ}$};
		
		\node[red] at (2.4,0.3)
		{$\cdots$};
		\node[red] at (3.2,0.3)
		{$\cdots$};
		\node[red] at (4.35,1.25)
		{$\cdots$};
		\node[red] at (6.4,1.25)
		{$\cdots$};
		\node[red] at (5.35,0.3)
		{$\cdots$};
		\node[red] at (7.3,1)
		{$\cdots$};
		%------------------------------------------------
		% 5. Clockwise-most ??-arc???
		%------------------------------------------------
		\draw[dark-green,very thick]
		(B)
		.. controls (2.0,0.65) and (7.8,1.15) ..
		(K);
		
		\node[dark-green] at (5.3,1.2)
		{$\gamma_M^{\circlearrowright}$};
		
		%------------------------------------------------
		% 6. ???
		%------------------------------------------------
		
		% ?????
		\foreach \v in {A,B,C,D,E,F,G,H,U,V,I,J,K}
		\fill[red] (\v) circle (2pt);
		
		% ?????
		\node[whitepoint] at (Wl) {};
		\node[whitepoint] at (Wr) {};
		
		% ??? tile/sector ????
		\foreach \w in {
			(2.25,-0.12),
			(3.15,1.72),
			(4.27,-0.12),
			(5.35,1.72),
			(6.32,-0.12),
			(7.43,1.65)
		}
		\node[whitepoint] at \w {};

	\end{tikzpicture}
	
	\caption{
		The permissible $\circ$-arc $\gamma_M^\circ$
		and its equivalent permissible clockwise-most
		$\rpoint$-arc $\gamma_M^{\circlearrowright}$.
	}
	\label{fig:equivalent-arc-representatives}
\end{figure}

Crossings of these representatives are oriented using a local triangle bounded by their segments and a dissection side.
One vertex of this triangle may be a red endpoint of an end segment.
A crossing is from $\beta$ to $\gamma$ when $\gamma$ follows $\beta$ clockwise around the triangle; this orientation is independent of the permitted triangle \cite[Lemma~5.4 and Definition~5.5]{BC24}.

\begin{definition}[{\!\!\cite[Definition~5.9]{BC24}}]\rm
\label{def:good-crossing}
A crossing from $\gamma_M^{\circlearrowright}$ to $\gamma_N^{\circlearrowright}$ is \defines{good} if its neighbourhood, including the terminal fan of $\gamma_N^{\circlearrowright}$ and the relevant label, is isotopic to the configuration in Figure~\ref{fig:good-crossing}.
Equivalently, it is the labelled terminal configuration of \cite[Figure~32]{BC24}, with our colour convention.
We write $X_{M,N}$ for the number of such oriented crossings.
\end{definition}

\begin{figure}[htbp]
\centering
\begin{tikzpicture}[scale=0.8,>=stealth]
\coordinate (p) at (0,0);
\coordinate (q) at (-0.85,2.8);
\draw[red,thick] (p)--(q) node[pos=.52,right,red] {$x_r$};
\draw[red,thick] (p)--(-2.9,.40) node[pos=.87,above,red] {$x_{r-1}$};
\draw[red,thick] (p)--(-2.9,-.42);
\draw[red,thick] (p)--(-2.55,-1.23);
\draw[red,thick] (p)--(-2.03,-1.78) node[pos=.93,below left,red] {$x_0$};
\node at (-1.7,-.64) {$\cdots$};
\draw[blue,very thick,->] (221:.70)
  arc[start angle=221,end angle=107,radius=.70];
\node[blue,right] at (.28,.43) {$a_1\cdots a_r$};
\draw[dark-green,thick] (-1.65,-2.12)
  .. controls (-1.69,.3) and (-1.68,1.96) .. (q);
\node[dark-green,left] at (-1.45,2.02) {$\gamma_N^{\circlearrowright}$};
\draw[black,thick] (-3,1.35)--(1.4,.96)
  node[right] {$\gamma_M^{\circlearrowright}$};
\node[below left] at (-1.51,1.22) {$c$};
\fill[red] (p) circle (2pt) node[below right,red] {$p$};
\fill[red] (q) circle (2pt) node[above,red] {$q$};
\end{tikzpicture}
\caption{A good crossing $c$ from $\gamma_M^{\circlearrowright}$ to $\gamma_N^{\circlearrowright}$.
The latter follows the labelled fan through $x_0,\ldots,x_{r-1}$ and ends at $q$ before crossing $x_r$.
The former crosses $x_r$.
The diagram specifies the local incidence pattern; the rest of the arcs is suppressed.}
\label{fig:good-crossing}
\end{figure}
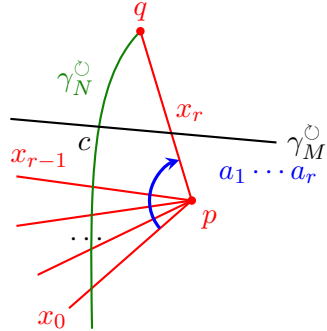

In this configuration, after choosing the direction of the string,
\[
w_N=w_\ell a_1\cdots a_{r-1},\quad r\geqslant2,
\]
where $w_\ell$ is nonzero, the displayed suffix is direct, and $a_1\cdots a_r$ is a label.
The arc $\gamma_M^{\circlearrowright}$ crosses $x_r$, the dissection arc indexed by $t(a_r)$.
These word conditions accompany the specified local configuration; by themselves they do not specify which crossing is good.
Write $\mathsf C(\gamma_M^{\circlearrowright},\gamma_N^{\circlearrowright})$ for the total number of oriented crossings.

\begin{proposition}[{\cite[Proposition~5.10]{BC24}}]
\label{prop:Hom-tau-crossing}
For any string modules $M$ and $N$, one has
\[
\dim_{\mathbf{k}}\Hom_A(M,\tau N) =\mathsf{C}(\gamma_M^{\circlearrowright}, \gamma_N^{\circlearrowright})-X_{M,N}.
\]
\end{proposition}

A \defines{generalised permissible arc} is either an arc of $\mathsf{P}$ or a nonzero permissible clockwise-most $\rpoint$-arc.
A collection of generalised permissible arcs is \defines{good} if all its oriented crossings, including self-crossings, are good.
An arc of $\mathsf{P}$ cannot cross another arc in a good collection.

\section{Gorenstein-projective support \texorpdfstring{$\tau$}{tau}-tilting modules}

In the gentle case, the non-projective indecomposable Gorenstein-projective modules are the right ideals $\alpha A$ attached to arrows on full-relation oriented cycles \cite[Theorem~2.5]{K15}, and support $\tau$-tilting pairs are described by non-crossing generalised dissections.
For general string algebras, perfect paths may have length greater than one and good crossings may occur, so we identify $pA$ with a string module and replace generalised dissections by maximal good collections.

\subsection{Perfect paths and Gorenstein-projective arcs}

For a perfect path $p$, let $w_p=a_1\cdots a_{m(p)}$ be the unique maximal, possibly trivial, path starting at $t(p)$ such that $pw_p\notin I$.
Then $pw_pa\in I$ for every arrow $a$ which can be appended to $w_p$.
Existence follows from the finite-dimensionality of $A$, and uniqueness from condition~\ref{S2}.

\begin{lemma}\label{lem:pA-string}
For every perfect path $p$, there is an isomorphism of right $A$-modules
\[
pA\cong M(w_p).
\]
\end{lemma}

\begin{proof}
Write $w_p=a_1\cdots a_m$.
{As} $A$ is monomial, $pA$ has a basis consisting of the nonzero paths beginning with $p$.
By condition~\ref{S2} and the maximality of $w_p$, these are exactly $p,\ pa_1,\ pa_1a_2,\ \ldots,\ pa_1\cdots a_m.$ The action of the arrows on this basis is precisely that defining the string module $M(w_p)$.
{Therefore} $pA\cong M(w_p)$.
\end{proof}
For each $i\in Q_0$, choose a string $w_i^{\mathrm{pr}}$ with $M(w_i^{\mathrm{pr}})\cong e_iA$, and write $\gamma_w^\circ$ for the corresponding permissible $\circ$-arc.
Set
\begin{equation}\label{eq:GP-circ-arcs}
\mathsf{GP}_{\circ}(\mathcal{S})
=
\{\gamma_{w_i^{\mathrm{pr}}}^{\circ}\mid i\in Q_0\}
\sqcup
\{\gamma_{w_p}^{\circ}\mid p\in\perpath(A)\}.
\end{equation}

\begin{corollary}\label{prop:GP-arcs}
The assignment $\gamma\mapsto M_\gamma$ induces a bijection
\[
\mathsf{GP}_{\circ}(\mathcal{S}) \longrightarrow \operatorname{ind}\bigl(\mathsf{Gproj}(A)\bigr).
\]
\end{corollary}

\begin{proof}
The projective modules {provide} the first set in~\eqref{eq:GP-circ-arcs}.
By Theorem~\ref{thm:perfect-path-GP}, every non-projective indecomposable Gorenstein-projective module is uniquely $pA$ for some $p\in\perpath(A)$.
Lemma~\ref{lem:pA-string} and Theorem~\ref{thm:geometric_modules} identify it with $\gamma_{w_p}^{\circ}$.
\end{proof}

For \(p\in\perpath(A)\), we call \(\gamma_{w_p}^{\circ}\) the Gorenstein-projective arc associated with \(p\).
Thus Corollary \ref{prop:GP-arcs} realizes \(\operatorname{ind}\bigl(\mathsf{Gproj}(A)\bigr)\) as the distinguished set \(\mathsf{GP}_{\circ}(\mathcal S)\) of permissible $\circ$-arcs on the labelled tiled surface.

\begin{remark} \rm \label{rem:projective-two-arcs}
The projective module $e_iA$ has two geometric realizations: the permissible $\circ$-arc $\gamma_{w_i^{\mathrm{pr}}}^{\circ}$ represents it in the module component of a support $\tau$-tilting pair, while the dissection arc $\mathbf{x}_i\in\mathsf P$ represents it in the projective component.
\end{remark}

In the label-free model of a gentle algebra, perfect paths are precisely the arrows lying on full-relation oriented cycles, and such cycles follow the boundaries of internal polygons.
{Therefore} the above construction {provides} a geometric realization of the classification in \cite[Theorem~2.5]{K15}.

For a general string algebra, minimal relations may be of type~\ref{R1} or type~\ref{R2}.
Consequently, a full-relation oriented cycle need not determine a perfect path sequence; the prefix and suffix conditions in Definition~\ref{def:perfect-path} are essential, as Example~\ref{ex:nonperfect-full-cycle} shows.
Notice also that the arc reading \(p\) represents \(M(p)\), whereas the Gorenstein-projective module \(pA\cong M(w_p)\) is represented by \(\gamma_{w_p}^{\circ}\).

\subsection{Geometric classification of support \texorpdfstring{$\tau$}{tau}-tilting pairs}

The algebraic conditions are now explicit: select pairwise non-isomorphic summands from $\{e_iA\}\sqcup\{pA\mid p\in\perpath(A)\}$, impose $\Hom_A(T,\tau T)=0$ and $\Hom_A(P,T)=0$, and require $|T|+|P|=|Q_0|$.
The following application of \cite[Theorem~5.13]{BC24} expresses these conditions on the fixed surface.

\begin{definition}\label{def:GP-support-collection} \rm
A \defines{Gorenstein-projective support collection} is a maximal good collection $\mathcal D$ of generalised permissible arcs such that the equivalence class of every permissible clockwise-most $\rpoint$-arc in $\mathcal D$ contains the unique permissible $\circ$-arc representative belonging to $\mathsf{GP}_{\circ}(\mathcal S)$.
\end{definition}

\begin{corollary}\label{cor:GP-tau-rigid}
Let $T=\bigoplus_{i=1}^{m}T_i$ be a basic Gorenstein-projective module.
Then $T$ is $\tau$-rigid if and only if every oriented crossing from $\gamma_{T_i}^{\circlearrowright}$ to $\gamma_{T_j}^{\circlearrowright}$ is good for all $1\leqslant i,j\leqslant m$.
\end{corollary}

\begin{proof}
{As} $\operatorname{Hom}_A(T,\tau T) = \bigoplus_{1\leqslant i,j\leqslant m} \operatorname{Hom}_A(T_i,\tau T_j)$, the assertion follows from Proposition~\ref{prop:Hom-tau-crossing}.
\end{proof}

\begin{proposition}\label{thm:GP-support-tau-geometry}
There is a bijection between the following two classes:
\begin{enumerate}
\item the isomorphism classes of basic Gorenstein-projective support \(\tau\)-tilting pairs; \item the Gorenstein-projective support collections on \((\mathcal S,\mathsf P,\mathsf L)\).
\end{enumerate}
More precisely, if \(T=\bigoplus_{i=1}^{m}T_i\) and \(P=\bigoplus_{j=1}^{r}e_{v_j}A\), then the corresponding collection is
\begin{equation}\label{eq:GP-pair-collection}
\mathcal D(T,P)
=
\{\gamma_{T_i}^{\circlearrowright}\mid 1\leqslant i\leqslant m\}
\cup
\{\mathbf{x}_{v_j}\mid 1\leqslant j\leqslant r\}.
\end{equation}
\end{proposition}

\begin{proof}
By \cite[Theorem 5.13]{BC24}, where the two components of a support pair are written in the opposite order, support $\tau$-tilting pairs are in bijection with maximal good collections of generalised permissible arcs.
Under this correspondence, the permissible clockwise-most $\rpoint$-arcs correspond to the indecomposable summands of the module part, while the selected arcs of $\mathsf P$ correspond to those of the projective part.

By Corollary~\ref{prop:GP-arcs}, the module part is Gorenstein-projective if and only if the equivalence class of every such $\rpoint$-arc contains the unique permissible $\circ$-arc representative belonging to $\mathsf{GP}_{\circ}(\mathcal S)$.
This is precisely Definition~\ref{def:GP-support-collection}.
{Therefore} the correspondence restricts to the required bijection.
\end{proof}

\subsection{A finite compatibility test}

We now translate the preceding geometric classification into a finite algebraic compatibility test that avoids explicit computation of the Auslander--Reiten translate.
This also yields bounds on the non-projective summands analogous to those in the gentle case \cite{LZH}.

\begin{proposition}\label{prop:perfect-successor-tau}
Let \(p\in\perpath(A)\) have perfect successor \(q\).
Then, for every \(A\)-module \(N\),
\begin{equation}\label{eq:perfect-successor-tau}
\Hom_A(N,\tau(pA))=0
\quad\Longleftrightarrow\quad
Ne_{t(p)}\xrightarrow{\,\cdot q\,}Ne_{t(q)}
\text{ is surjective}.
\end{equation}
In particular, no \(\tau\)-rigid module contains \(pA\) together with \(qA\) or \(e_{t(q)}A\), and \(pA\) is not \(\tau\)-rigid if \(t(p)=t(q)\).
\end{proposition}

\begin{proof}
Perfectness {provides} a minimal projective presentation
\[
e_{t(q)}A\xrightarrow{\lambda_q}e_{t(p)}A \xrightarrow{\lambda_p}pA\longrightarrow0, \quad \lambda_q(x)=qx,\quad \lambda_p(x)=px.
\]
Indeed, \(\ker\lambda_p=qA\), while \(e_{t(p)}A\twoheadrightarrow pA\) and \(e_{t(q)}A\twoheadrightarrow qA\) are projective covers; the latter map is the corestriction of \(\lambda_q\) to its image.
Under the identifications \(\Hom_A(e_iA,N)\cong Ne_i\), the map induced by \(\lambda_q\) is right multiplication by \(q\).
Proposition \(\ref{prop:projective-presentation-tau-rigid}\) therefore {provides} \eqref{eq:perfect-successor-tau}.

For \(N=qA\), every nonzero path in the image of multiplication by \(q\) has length greater than \(\ell(q)\), so the image does not contain \(q\in qAe_{t(q)}\).
For \(N=e_{t(q)}A\), the image consists of paths of positive length and {therefore} does not contain \(e_{t(q)}\).
Thus \(qA\) and \(e_{t(q)}A\) cannot occur in a \(\tau\)-rigid module containing \(pA\).
Finally, if \(t(p)=t(q)\), multiplication by \(q\) on \(pA\) does not contain the shortest basis path \(p\) in its image.
{Therefore} \(\Hom_A(pA,\tau(pA))\neq0\).
\end{proof}

\begin{remark} \rm
If \(N\) is a string module, Proposition~\ref{prop:Hom-tau-crossing} and \eqref{eq:perfect-successor-tau} show that \(Ne_{t(p)}\xrightarrow{\cdot q}Ne_{t(q)}\) is surjective if and only if every oriented crossing from \(\gamma_N^{\circlearrowright}\) to \(\gamma_{pA}^{\circlearrowright}\) is good.
Thus failure of surjectivity is detected geometrically by a non-good oriented crossing.
\end{remark}

The perfect-successor map on \(\perpath(A)\) is a permutation, {as} every perfect path has a unique successor and predecessor.
Let \(m_1,\ldots,m_c\) be the lengths of its cycles.

\begin{corollary}\label{cor:perfect-cycle-bound}
Let \(T\) be a basic Gorenstein-projective \(\tau\)-rigid module, and let \(T_{\mathrm{np}}\) be its non-projective part.
Then
\[
|T_{\mathrm{np}}| \leqslant \sum_{j=1}^{c}\left\lfloor\frac{m_j}{2}\right\rfloor.
\]
In particular, a fixed point of the perfect-successor permutation cannot correspond to a summand of \(T\).
\end{corollary}

\begin{proof}
By Proposition~\ref{prop:perfect-successor-tau}, no two consecutive perfect paths in a successor cycle can correspond to summands of \(T_{\mathrm{np}}\).
For a cycle of length \(m\), let \(x_1,\ldots,x_m\in\{0,1\}\) indicate the selected paths.
Then \(x_i+x_{i+1}\leqslant1\), with cyclic indices.
Summing {provides} \(2\sum_i x_i\leqslant m\), so at most \(\lfloor m/2\rfloor\) paths can be selected.
Summing over all cycles proves the inequality.
For \(m=1\), the same inequality {provides} \(x_1=0\).
\end{proof}

For a gentle algebra, the cycles of the perfect-successor permutation correspond precisely to the full-relation cycles, so Corollary \(\ref{cor:perfect-cycle-bound}\) recovers the bound in \cite[Corollary~5.3]{LZH}.
The bound is only necessary: a selection attaining it need not yield a \(\tau\)-rigid module or extend to a Gorenstein-projective support \(\tau\)-tilting pair.

\begin{corollary}\label{prop:single-replacement}
Let \(\alpha,\beta\in Q_1\cap\perpath(A)\) such that \((\alpha,\beta)\) is a perfect pair.
Suppose that \(s(\beta)\neq t(\beta)\) and that \(\beta\) is the unique arrow in \(Q_1\) with target \(t(\beta)\).
Then
\[
\alpha A\oplus\bigoplus_{i\neq t(\beta)}e_iA
\]
is a basic Gorenstein-projective \(\tau\)-tilting module.
\end{corollary}

\begin{proof}
Set \(j=t(\beta)\) and \(T=\alpha A\oplus\bigoplus_{i\neq j}e_iA\).
For \(i\neq j\), let \(w\) be a path-basis element of \(e_iAe_j\).
{As} \(i\neq j\), the path \(w\) has positive length.
As \(\beta\) is the unique arrow with target \(j\), we have \(w=u\beta\) for some \(u\in e_iAe_{s(\beta)}\).
{Therefore} \(e_iAe_{s(\beta)}\xrightarrow{\cdot\beta}e_iAe_j\) is surjective.

Now let \(w\) be a path-basis element of \(\alpha Ae_j\).
{As} \(t(\alpha)=s(\beta)\neq j\), we have \(w\neq\alpha\).
Thus \(w\) has positive length after its initial subpath \(\alpha\), and the uniqueness of \(\beta\) {provides} \(w=u\beta\) for some \(u\in\alpha Ae_{s(\beta)}\).
Therefore \(\alpha Ae_{s(\beta)}\xrightarrow{\cdot\beta}\alpha Ae_j\) is also surjective.

It follows that \(Te_{s(\beta)}\xrightarrow{\cdot\beta}Te_j\) is surjective.
{As} \(\beta\) is the perfect successor of \(\alpha\), Proposition \(\ref{prop:perfect-successor-tau}\) {provides} \(\Hom_A(T,\tau(\alpha A))=0\).
Moreover, \(\tau(e_iA)=0\) for every \(i\), and {therefore} \(\Hom_A(T,\tau T)=0\).
Thus \(T\) is \(\tau\)-rigid.

{As} \(\alpha\) is a perfect path, \(\alpha A\) is indecomposable, non-projective and Gorenstein-projective.
Together with the remaining pairwise non-isomorphic indecomposable projective summands, this shows that \(T\) is a basic Gorenstein-projective module with \(|Q_0|\) indecomposable summands.
{Therefore} \(T\) is \(\tau\)-tilting.
\end{proof}

For gentle algebras, Corollary~\ref{prop:single-replacement} recovers the constructive implication of \cite[Theorem~5.2]{LZH}.
For general string algebras, it {provides} only a sufficient condition.
The assumption \(s(\beta)\neq t(\beta)\) cannot be omitted.
Indeed, for \(A=\mathbf{k}[t]/(t^2)\), the pair \((t,t)\) is perfect and \(t\) is its own perfect successor.
{Therefore} \(tA\) is not \(\tau\)-rigid by Proposition~\ref{prop:perfect-successor-tau}.

For \(p\in\perpath(A)\), let \(q_p\) be its perfect successor and define
\[
\mathcal O(p)= \left\{ i\in Q_0\ \middle|\ e_iAe_{t(p)} \xrightarrow{\,\cdot q_p\,} e_iAe_{t(q_p)} \text{ is not surjective} \right\}.
\]
For \(U\subseteq\perpath(A)\), {let} \(\mathcal O(U)=\bigcup_{p\in U}\mathcal O(p)\).

We call a basic Gorenstein-projective \(\tau\)-tilting module non-trivial if it has at least one non-projective indecomposable direct summand.

\begin{proposition}\label{thm:nontrivial-GP-tau-existence}
The algebra \(A\) admits a non-trivial basic Gorenstein-projective \(\tau\)-tilting module if and only if there exists a nonempty set \(U\subseteq\perpath(A)\) such that
\[
rAe_{t(p)} \xrightarrow{\,\cdot q_p\,} rAe_{t(q_p)} \quad\text{is surjective for all }p,r\in U, \quad |\mathcal O(U)|=|U|.
\]
In this case,
\[
T_U= \bigoplus_{p\in U}pA \oplus \bigoplus_{i\notin\mathcal O(U)}e_iA
\]
is a basic Gorenstein-projective \(\tau\)-tilting module.
Conversely, every non-trivial basic Gorenstein-projective \(\tau\)-tilting module is of this form.
\end{proposition}

\begin{proof}
Suppose first that \(U\) satisfies the stated conditions.
By Proposition~\ref{prop:perfect-successor-tau}, the first condition shows that \(\bigoplus_{p\in U}pA\) is \(\tau\)-rigid.
If \(i\notin\mathcal O(U)\), the same proposition {provides} \(\Hom_A(e_iA,\tau(pA))=0\) for every \(p\in U\).
{As} projective modules have zero Auslander--Reiten translate, \(T_U\) is \(\tau\)-rigid.
Moreover, \(|T_U|=|U|+|Q_0|-|\mathcal O(U)|=|Q_0|\).
Here distinct perfect paths determine pairwise non-isomorphic indecomposable Gorenstein-projective modules by Theorem~\ref{thm:perfect-path-GP}, so this count has no repetitions.
{Therefore} \(T_U\) is \(\tau\)-tilting.
It is Gorenstein-projective because every \(pA\) is Gorenstein-projective and all its remaining summands are projective.

Conversely, let $T=\bigoplus_{p\in U}pA\oplus\bigoplus_{i\in V}e_iA$ be a basic Gorenstein-projective $\tau$-tilting module with $U\neq\varnothing$.
By Proposition~\ref{prop:perfect-successor-tau}, the required surjectivity holds for all $p,r\in U$, and $\mathcal O(U)\subseteq Q_0\setminus V$.

It follows from the definition of $\mathcal O(U)$ that $T'=\bigoplus_{p\in U}pA\oplus\bigoplus_{i\notin \mathcal O(U)}e_iA$ is $\tau$-rigid.
{As} $V\subseteq Q_0\setminus \mathcal O(U)$, the module $T$ is a direct summand of $T'$.
As $T$ is $\tau$-tilting, Proposition~\ref{prop:AIR-maximality} {provides} $T'=T$.
{Therefore} $\mathcal O(U)=Q_0\setminus V$, and therefore $|\mathcal O(U)|=|U|$.
\end{proof}

Proposition \ref{thm:nontrivial-GP-tau-existence} {provides} an existence criterion for non-trivial Gorenstein-projective \(\tau\)-tilting modules.
We now encode all Gorenstein-projective support \(\tau\)-tilting pairs in a finite compatibility graph.

Let \(n=|Q_0|\), and let \(\mathcal M_A\) be the set of isomorphism classes of the \(\tau\)-rigid modules of the form \(e_iA\), \(i\in Q_0\), or \(pA\), \(p\in{\perpath(A)}\).
For each \(i\in Q_0\), introduce a formal negative vertex \(i^-\), representing \(e_iA\) in the projective part of a support pair.

Define a graph \(\mathcal G_A\) with vertex set \(\mathcal M_A\sqcup\{i^-\mid i\in Q_0\}\).
{Distinct \(X,Y\in\mathcal M_A\) are joined if
\[
\Hom_A(X,\tau Y)=0=\Hom_A(Y,\tau X).
\]
Equivalently, whenever \(\{X,Y\}=\{pA,Z\}\) with \(p\in\perpath(A)\), multiplication by the perfect successor \(q_p\) induces a surjection \(Ze_{t(p)}\to Ze_{t(q_p)}\); when both vertices are non-projective, this condition is imposed in both directions.}
Join \(i^-\) to \(X\in\mathcal M_A\) if \(Xe_i=0\), and join every pair of distinct negative vertices.

For a clique \(\mathcal C\) of \(\mathcal G_A\), set
\[
M_{\mathcal C}=\bigoplus_{X\in\mathcal C\cap\mathcal M_A}X, \quad P_{\mathcal C}=\bigoplus_{i^-\in\mathcal C}e_iA.
\]

\begin{theorem}\label{thm:finite-compatibility-graph}
The assignment \(\mathcal C\mapsto(M_{\mathcal C},P_{\mathcal C})\) induces a bijection between the \(n\)-element sets of pairwise adjacent vertices of \(\mathcal G_A\) and the isomorphism classes of basic Gorenstein-projective support \(\tau\)-tilting pairs.
\end{theorem}

\begin{proof}
By Proposition~\ref{prop:perfect-successor-tau}, two distinct vertices \(X,Y\in\mathcal M_A\) are adjacent precisely when \(\Hom_A(X,\tau Y)=\Hom_A(Y,\tau X)=0\).
{Therefore} the module vertices in \(\mathcal C\) are pairwise adjacent if and only if \(M_{\mathcal C}\) is \(\tau\)-rigid.

{As} \(\Hom_A(e_iA,X)\cong Xe_i\), adjacency between negative and module vertices is equivalent to \(\Hom_A(P_{\mathcal C},M_{\mathcal C})=0\).
Thus \((M_{\mathcal C},P_{\mathcal C})\) is a basic Gorenstein-projective \(\tau\)-rigid pair.
Conversely, the classification of indecomposable Gorenstein-projective modules shows that every such pair determines a unique set of pairwise adjacent vertices of \(\mathcal G_A\).

Finally, \(\lvert M_{\mathcal C}\rvert+\lvert P_{\mathcal C}\rvert =\lvert\mathcal C\rvert\).
Therefore \((M_{\mathcal C},P_{\mathcal C})\) is support \(\tau\)-tilting if and only if \(\lvert\mathcal C\rvert=n\).
\end{proof}
Theorem~\ref{thm:finite-compatibility-graph} {provides} a finite algorithm from \((Q,I)\).
First list all nonzero paths and use conditions~\ref{P1}--\ref{P3} to find the perfect paths and their successors.
Then construct the required path bases, check which multiplication maps are surjective, and list all \(n\)-element sets of pairwise adjacent vertices of \(\mathcal G_A\).
This process terminates because \(A\) has only finitely many nonzero paths.

A maximal set of pairwise adjacent vertices may have fewer than \(n\) elements.
Thus maximality alone is not enough.
The construction does not use a Gorenstein-projective Bongartz completion, which may not exist in general \cite{XZ24}.

\section{Endomorphism algebras}

Throughout this section, \(T\) denotes a nonzero basic Gorenstein-projective \(A\)-module; the trivial case \(T=0\) is omitted.

\subsection{Algebraic presentations}

Let
\[
G_A=\bigoplus_{i\in Q_0}e_iA\oplus \bigoplus_{p\in{\perpath(A)}}pA
\]
be a basic additive generator of \(\mathsf{Gproj}(A)\).
The Cohen--Macaulay Auslander algebra of \(A\) is \(A^{\mathrm{CMA}}=\End_A(G_A)\).
It is finite-dimensional by Corollary~\ref{cor:CM-finite}, and its bound quiver is described in \cite[Construction~3.14 and Theorem~3.15]{LZ24}.
Its vertices correspond to the modules \(e_iA\), \(i\in Q_0\), and \(pA\), \(p\in{\perpath(A)}\).
If the G-condition fails, \(A^{\mathrm{CMA}}\) may have non-monomial relations.

\begin{lemma}\label{prop:End-corner}
There exists an idempotent \(e_T\in A^{\mathrm{CMA}}\) such that
\[
\End_A(T)\cong e_TA^{\mathrm{CMA}}e_T.
\]
In particular, \(\End_A(T)\) is isomorphic to a corner algebra of \(A^{\mathrm{CMA}}\).
\end{lemma}

\begin{proof}
{As} \(T\) is basic and Gorenstein-projective, it may be identified with a direct summand of \(G_A\).
Write \(G_A=T\oplus T'\), and let \(\iota:T\to G_A\) and \(\pi:G_A\to T\) be the canonical inclusion and projection.
{As} \(\pi\iota=1_T\), the element \(e_T=\iota\pi\) is idempotent.
The map
\[
\End_A(T)\longrightarrow e_T\End_A(G_A)e_T,\quad f\longmapsto\iota f\pi,
\]
is an algebra isomorphism with inverse \(g\mapsto\pi g\iota\).
\end{proof}

Write \(T=\bigoplus_{i=1}^mT_i\), where the \(T_i\) are pairwise non-isomorphic indecomposable modules.
For \(X,Y\in\add T\), let \(\rad_A(X,Y)\) denote the space of radical maps from \(X\) to \(Y\), set
\begin{equation}\label{eq:relative-radical-square}
\rad_{\add T}^{2}(T_j,T_i)
=\sum_{k=1}^m
\rad_A(T_k,T_i)\circ\rad_A(T_j,T_k).
\end{equation}
Thus \(\rad_{\add T}^{2}(T_j,T_i)\) consists of sums of compositions of two radical maps through summands of \(T\).

For each \(i,j\), choose a basis of \(\rad_A(T_j,T_i)/\rad_{\add T}^{2}(T_j,T_i)\) and fix representatives in \(\rad_A(T_j,T_i)\).
Let \(Q_T\) be the quiver with vertices \(1,\ldots,m\) and one arrow \(a:i\to j\) for each chosen representative \(f_a:T_j\to T_i\).

Let \(\varepsilon_i\in B=\End_A(T)\) be the primitive idempotent corresponding to \(T_i\).
Under the identification \(\varepsilon_iB\varepsilon_j\cong\Hom_A(T_j,T_i)\), the assignments
\[
i\longmapsto\varepsilon_i, \quad a\longmapsto f_a
\]
define an algebra homomorphism \(\Phi_T:\mathbf{k}Q_T\to B\).
Set \(I_T=\ker\Phi_T\).

\begin{lemma}\label{thm:bound-quiver-End}
The map \(\Phi_T\) is surjective, the ideal \(I_T\) is admissible, and
\[
B\cong\mathbf{k}Q_T/I_T.
\]
In particular, \(Q_T\) is the Gabriel quiver of \(B\), and an arrow \(i\to j\) represents an irreducible map \(T_j\to T_i\).
\end{lemma}

\begin{proof}
Under the natural identification \(\varepsilon_iB\varepsilon_j\cong\Hom_A(T_j,T_i)\), one has
\[
\varepsilon_i\rad B\varepsilon_j \cong\rad_A(T_j,T_i), \quad \varepsilon_i\rad^2B\varepsilon_j \cong \sum_{k=1}^{m} \rad_A(T_k,T_i)\circ\rad_A(T_j,T_k) = \rad_{\add T}^{2}(T_j,T_i).
\]
Thus \(Q_T\) is the Gabriel quiver of \(B\).
The chosen arrow representatives lift a basis of \(\rad B/\rad^2B\); {as} \(\rad B\) is nilpotent, they generate \(B\) together with the primitive idempotents.
{Therefore} \(\Phi_T\) is surjective.

Let \(J\) be the arrow ideal of \(\mathbf{k}Q_T\).
The choice of the vertices and arrows {provides} \(I_T\subseteq J^2\).
If \((\rad B)^N=0\), then \(J^N\subseteq I_T\).
Therefore \(I_T\) is admissible, and \(B\cong\mathbf{k}Q_T/I_T\).
\end{proof}

Thus the relations in \(I_T\) are precisely the linear combinations of paths whose corresponding compositions of the maps \(f_a\) are zero.

The graph-map bases in Theorem~\ref{thm:homs} make Lemma~\ref{thm:bound-quiver-End} explicit.
{As} the \(T_i\) are pairwise non-isomorphic and indecomposable, \(\rad_A(T_j,T_i)=\Hom_A(T_j,T_i)\) for \(i\neq j\).
In \(\End_A(T_i)\), the identity is the only graph-map basis element outside the radical, so the remaining graph maps form a basis of \(\rad\End_A(T_i)\).
The arrows of \(Q_T\) are therefore obtained from bases of \(\rad_A(T_j,T_i)/\rad_{\add T}^{2}(T_j,T_i)\), and the relations are obtained by composing their graph-map representatives.
{As} \(\rad\End_A(T)\) is nilpotent, only finitely many paths need to be considered.
Paths of different lengths must be treated together, {as} the relations need not be homogeneous in path length; see Example~\ref{ex:mixed-length-relation}.

We next describe the corner algebra determined by a set of vertices.
The construction retains paths passing through vertices outside this set.

\begin{definition}\label{def:S-segment}
Let \(\Gamma=\mathbf{k}Q^\Gamma/J\) be a string algebra and let \(S\subseteq Q^\Gamma_0\).
A path \(u\notin J\) of positive length is called an \(S\)-segment if its endpoints lie in \(S\) and no internal vertex lies in \(S\).
Let \(Q_S\) be the quiver with vertex set \(S\) and an arrow \([u]:s(u)\to t(u)\) for each \(S\)-segment \(u\).
Let \(J_S\) be the monomial ideal generated by all paths \([u_1]\cdots[u_r]\), \(r\geqslant2\), such that \(u_1\cdots u_r\in J\).
\end{definition}

\begin{proposition}\label{prop:corner-string-compression}
Let $e=\sum_{v\in S}e_v$.
The assignment $[u]\mapsto\overline{u}$ induces an algebra isomorphism $e\Gamma e\cong\mathbf{k}Q_S/J_S.$ Moreover, $\mathbf{k}Q_S/J_S$ is a string algebra.
\end{proposition}

\begin{proof}
Every nonzero path in \(\Gamma\) whose endpoints lie in \(S\) decomposes uniquely, at its successive visits to \(S\), into \(S\)-segments.
{Therefore} \([u]\mapsto\overline{u}\) {provides} a bijection between the path bases of \(\mathbf{k}Q_S/J_S\) and \(e\Gamma e\), and this bijection preserves multiplication.
Therefore \(e\Gamma e\cong\mathbf{k}Q_S/J_S\).
{As} \(\Gamma\) is finite-dimensional, all sufficiently long paths in \(Q_S\) belong to \(J_S\); {therefore} \(J_S\) is admissible and monomial.

An \(S\)-segment starting at a fixed vertex is determined by its first arrow, {as} condition~\ref{S2} allows at most one nonzero continuation.
By condition~\ref{S1}, at most two arrows of \(Q_S\) start at each vertex, and dually at most two end there.
Moreover, condition~\ref{S2} allows every arrow of \(Q_S\) at most one nonzero continuation on either side.
Thus \(\mathbf{k}Q_S/J_S\) is a string algebra.
\end{proof}

\begin{remark} \rm \label{rem:finite-compression}
Compression cannot be replaced by taking a full subquiver.
For example, if \(\Gamma=\mathbf{k}(1\xrightarrow{a}2\xrightarrow{b}3)\) and \(e=e_1+e_3\), then the full subquiver on \(\{1,3\}\) has no arrows, whereas \(e\Gamma e\) contains the arrow \(1\to3\) represented by the path \(ab\).

\end{remark}

\subsection{Geometric construction under the G-condition}

Recall from Definition~\ref{def:G-condition} that $A$ satisfies the G-condition precisely when every relation cycle supporting a perfect path sequence is gentle.
In particular, a string algebra without relation cycles satisfies the G-condition.

\begin{lemma}[{\cite[Corollary~4.2(2)]{LZ24}}]\label{thm:CMA-string-G}
The Cohen--Macaulay Auslander algebra $A^{\mathrm{CMA}}$ is a string algebra if and only if $A$ satisfies the G-condition.
\end{lemma}

\begin{corollary}\label{cor:End-string-G}
Assume that $A$ satisfies the G-condition.
For every basic Gorenstein-projective module $T\neq0$, the algebra $\End_A(T)$ is the string algebra obtained from the bound quiver of $A^{\mathrm{CMA}}$ by the $S_T$-segment construction of Definition~\ref{def:S-segment}, where $S_T$ is the set of vertices corresponding to the indecomposable summands of $T$.
\end{corollary}

\begin{proof}
Combine Lemma~\ref{prop:End-corner}, Lemma~\ref{thm:CMA-string-G}, and Proposition~\ref{prop:corner-string-compression}.
\end{proof}

We first describe the bound quiver of \(A^{\mathrm{CMA}}\).
Under the G-condition, every perfect path is an arrow, and each right ideal \(\alpha A\) corresponds to a new vertex \(v_\alpha\).
This is the perfect-arrow case of \cite[Section~3 and Corollary~4.2]{LZ24}.

\begin{lemma}\label{lem:perfect-arrow-endpoints}
Assume the G-condition.
Let \(\alpha\) be a perfect arrow, and let \(\gamma\) and \(\beta\) be its perfect predecessor and successor, respectively.
Then \(\alpha\beta\) is the only minimal relation beginning with \(\alpha\), and \(\gamma\alpha\) is the only minimal relation ending with \(\alpha\).
\end{lemma}

\begin{proof}
Let \(\alpha v\) be a minimal relation, where \(v\) is a nonzero path.
By condition~\ref{P2}, \(\beta\) is a prefix of \(v\).
{Therefore} \(\alpha v\) contains the relation \(\alpha\beta\), and minimality {provides} \(v=\beta\).
The second assertion follows dually from condition~\ref{P3}.
\end{proof}

Assume the G-condition and set \(R={\perpath(A)}\subseteq Q_1\).
Construct \(Q^\Gamma\) from \(Q\) by replacing each \(\alpha\in R\) with $s(\alpha)\xrightarrow{\alpha^-}v_\alpha \xrightarrow{\alpha^+}t(\alpha),$ and leaving every arrow outside \(R\) unchanged.
Thus $Q^\Gamma_0=Q_0\sqcup\{v_\alpha\mid\alpha\in R\}.$ For a path \(p\) in \(Q\), let \(p^*\) be obtained by replacing every \(\alpha\in R\) with \(\alpha^-\alpha^+\).
Let \(J\) be generated by
\begin{enumerate}[label=(E\arabic*)]

\item  \(\alpha^+\beta^-\) for each \(\alpha\in R\), where \(\beta\) is the perfect successor of \(\alpha\); \label{E1}

\item \(r^*\) for each minimal relation \(r\) of \(A\) which is not of the form \(\alpha\beta\) occurring in \ref{E1}.\label{E2}

\end{enumerate}

For \(i\in Q_0\), set \(G_i=e_iA\), and for \(\alpha\in R\), set \(G_{v_\alpha}=\alpha A\).
Then \(G_A=\bigoplus_{x\in Q^\Gamma_0}G_x\).
Define \(\Phi:\mathbf{k}Q^\Gamma\to\End_A(G_A)\) as follows.
Each vertex is sent to the corresponding primitive idempotent.
An unchanged arrow \(a:i\to j\) is sent to the map \(e_jA\to e_iA\), \(x\mapsto ax\), and
\[
\Phi(\alpha^-): \alpha A\hookrightarrow e_{s(\alpha)}A, \quad \Phi(\alpha^+): e_{t(\alpha)}A\twoheadrightarrow\alpha A,\quad x\mapsto\alpha x.
\]

Figure~\ref{fig:local-perfect-arrow-splitting} illustrates
the local geometric realization of the perfect-arrow
splitting construction.

\begin{figure}[htbp]
	\centering
	\begin{tikzpicture}[scale=0.35,>=stealth]
		
		% Indices are cyclic: x_7=x_1.
		% At corner i the clockwise arrow is alpha_i: x_i -> x_{i+1}.
		\begin{scope}[shift={(-10,0)}]
			\foreach \i/\ang in {1/90,2/150,3/210,4/270,5/330,6/390}
			\coordinate (P\i) at (\ang:6);
			
			\draw[red,line width=1pt]
			(P1)--(P2)--(P3)--(P4)--(P5)--(P6)--cycle;
			
			\foreach \i/\ang/\j in {
				1/90/2,2/150/3,3/210/4,4/270/5,5/330/6,6/390/1
			} {
				\begin{scope}[shift={(P\i)},rotate={\ang-90}]
					\draw[dark-green,line width=1.2pt,->]
					(330:1.65) arc[start angle=330,end angle=210,radius=1.65];
				\end{scope}
				\node[dark-green,font=\scriptsize] at (\ang:3.8) {$\alpha_{\i}$};
				\node[red,font=\scriptsize] at ({\ang+30}:5.95) {$x_{\j}$};
			}
			
			\foreach \i in {1,...,6}
			\draw[thick,red,fill=red] (P\i) circle (0.2cm);
			\draw[thick,black,fill=white] (0,0) circle (0.2cm);

		\end{scope}
		
		\node at (0,0) {$\Longrightarrow$};
		
		\begin{scope}[shift={(10,0)}]
			\foreach \i/\ang in {1/90,2/150,3/210,4/270,5/330,6/390} {
				\coordinate (Q\i) at (\ang:6);
				\coordinate (R\i) at (\ang:2);
			}
			
			\draw[black,line width=1.5pt,fill=gray!50] (0,0) circle (2cm);
			\draw[red,line width=1pt]
			(Q1)--(Q2)--(Q3)--(Q4)--(Q5)--(Q6)--cycle;
			
			\foreach \i in {1,...,6}
			\draw[red,line width=1pt] (Q\i)--(R\i);
			
			\foreach \i/\ang/\j in {
				1/90/2,2/150/3,3/210/4,4/270/5,5/330/6,6/390/1
			} {
				\begin{scope}[shift={(Q\i)},rotate={\ang-90}]
					\draw[dark-green,line width=1.2pt,->]
					(330:1.65) arc[start angle=330,end angle=270,radius=1.65];
					\draw[dark-green,line width=1.2pt,->]
					(270:1.65) arc[start angle=270,end angle=210,radius=1.65];
					\node[dark-green,font=\scriptsize] at (1.25,-2.1)
					{$\alpha_{\i}^{-}$};
					\node[dark-green,font=\scriptsize] at (-1.25,-2.1)
					{$\alpha_{\i}^{+}$};
				\end{scope}
				\node[red,font=\scriptsize] at ({\ang+30}:5.95) {$x_{\j}$};
				
			}
			
			\foreach \i/\ang in {1/74,2/134,3/226,4/254,5/314,6/46}
			\node[red,font=\scriptsize] at (\ang:3.05) {$v_{\alpha_{\i}}$};
			
			\foreach \i in {1,...,6} {
				\draw[thick,red,fill=red] (Q\i) circle (0.2cm);
				\draw[thick,red,fill=red] (R\i) circle (0.2cm);
			}
			\foreach \ang in {60,0,-60,-120,180,120}
			\draw[thick,black,fill=white] (\ang:2) circle (0.2cm);

		\end{scope}
		
	\end{tikzpicture}
	\caption{
		Splitting an internal polygon whose angles correspond to
		perfect arrows. Each angle $\alpha_i:x_i\to x_{i+1}$ is
		replaced by two consecutive angles
		$\alpha_i^-:x_i\to v_{\alpha_i}$ and
		$\alpha_i^+:v_{\alpha_i}\to x_{i+1}$, with indices taken modulo $6$.
		The white puncture becomes a boundary component with alternating
		red and white marked points; the shaded disk is removed.
	}
	\label{fig:local-perfect-arrow-splitting}
\end{figure}
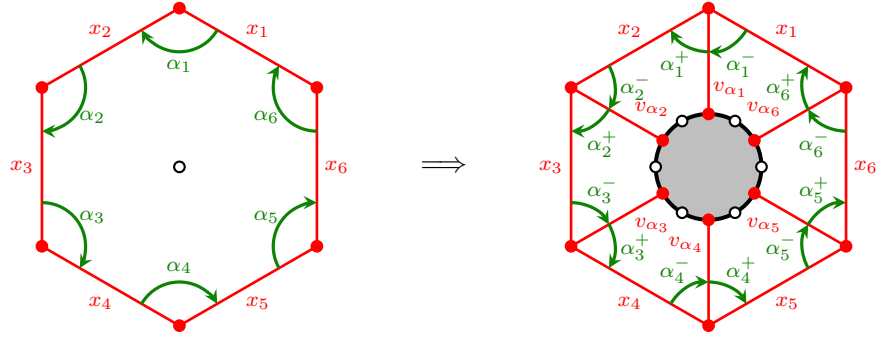

The following proposition is proved by Chen and Lu in \cite{CL2017, CL2019}.
In order to facilitate readers' understanding of the details, we still provide a proof for this proposition. It should be noted that this proof idea comes from Chen and Lu, and it is not a new proof we have provided.

\begin{proposition}\label{prop:perfect-arrow-splitting}
The homomorphism \(\Phi\) induces an algebra isomorphism
\[
\overline{\Phi}:\mathbf{k}Q^\Gamma/J \xrightarrow{\ \cong\ }A^{\mathrm{CMA}}.
\]
Moreover, \(J\) is admissible and monomial, and \((Q^\Gamma,J)\) satisfies the string conditions.
\end{proposition}

\begin{proof}
For each \(\alpha\in R\), let \(\beta\) be its perfect successor.
The perfect-pair exact sequence
\[
0\longrightarrow\beta A \longrightarrow e_{t(\alpha)}A \xrightarrow{\lambda_\alpha}\alpha A \longrightarrow0, \quad \lambda_\alpha(x)=\alpha x,
\]
{provides} \(\Phi(\alpha^+\beta^-)=0\).
If \(r\) is any other minimal relation of \(A\), then \(\Phi(r^*)=0\).
{Therefore} \(J\subseteq\ker\Phi\), and \(\Phi\) induces an algebra homomorphism \(\overline{\Phi}:\mathbf{k}Q^\Gamma/J\to\End_A(G_A)\).

Fix \(x,y\in Q^\Gamma_0\).
For a path \(u\) from \(x\) to \(y\), prepend \(\delta^-\) if \(x=v_\delta\), append \(\alpha^+\) if \(y=v_\alpha\), and then contract every pair \(\eta^-\eta^+\) to \(\eta\).
Denote the resulting path in \(Q\) by \(c(u)\).
We claim that \(u\notin J\) if and only if \(c(u)\notin I\).

If \(u\in J\), then \(u\) contains either some \(\alpha^+\beta^-\) or some \(r^*\), so \(c(u)\) contains the corresponding relation of \(A\).
Conversely, suppose that \(c(u)\in I\).
{As} \(I\) is monomial, \(c(u)\) contains a minimal relation.
If this relation does not meet an arrow added at an endpoint of \(u\), then \(u\) contains the corresponding generator of \(J\).
Otherwise, Lemma~\ref{lem:perfect-arrow-endpoints} shows that it is a perfect-pair relation, so the corresponding \(\alpha^+\beta^-\) occurs in \(u\).
This proves the claim.

{As} \(J\) is monomial, the paths \(u\notin J\) from \(x\) to \(y\) form a basis of \(e_x(\mathbf{k}Q^\Gamma/J)e_y\).
We compare this basis with a path basis of \(\Hom_A(G_y,G_x)\).

First suppose that \(y=j\in Q_0\).
{As} \(\Hom_A(e_jA,G_x)\cong G_xe_j\), its standard basis consists of the nonzero paths ending at \(j\), and these paths must also begin with \(\delta\) when \(x=v_\delta\).
The correspondence \(u\mapsto c(u)\) identifies this basis with the paths \(u\notin J\) from \(x\) to \(y\).

Now suppose that \(y=v_\alpha\), and let \(\beta\) be the perfect successor of \(\alpha\).
A homomorphism \(f:\alpha A\to G_x\) is determined by \(z=f(\alpha)\in G_xe_{t(\alpha)}\), subject to \(z\beta=0\).
{As} \(A\) is monomial, this condition may be checked on the path basis.
By condition~\ref{P3}, the nonzero path-basis elements annihilated by \(\beta\) are precisely those ending with \(\alpha\); when \(x=v_\delta\), they must also begin with \(\delta\).
Thus \(u\mapsto c(u)\) again identifies these basis elements with the paths \(u\notin J\) from \(x\) to \(y\).
When \(x=y=v_\alpha\), the trivial path corresponds to \(1_{\alpha A}\).

Therefore \(\overline{\Phi}\) restricts to an isomorphism \(e_x(\mathbf{k}Q^\Gamma/J)e_y\cong\Hom_A(G_y,G_x)\) for every \(x,y\in Q^\Gamma_0\).
{Therefore} \(\overline{\Phi}\) is an algebra isomorphism.

It remains to verify the string conditions.
By construction, \(J\) is monomial and generated by paths of length at least two.
{As} \(\mathbf{k}Q^\Gamma/J\cong A^{\mathrm{CMA}}\) is finite-dimensional, \(J\) is admissible.
Splitting preserves the in- and out-degrees of every original vertex, while each new vertex has one incoming and one outgoing arrow.
Thus condition~\ref{S1} holds.

At an original vertex, the length-two paths in \(Q^\Gamma\) correspond to those in \(Q\).
By Lemma~\ref{lem:perfect-arrow-endpoints}, every quadratic relation containing a perfect arrow is the product of a perfect pair.
Consequently, the zero and nonzero continuations at the original vertices are precisely those inherited from \(A\), while at each new vertex there is only one possible continuation in either direction.
{Therefore} condition~\ref{S2} also holds.
\end{proof}

In particular, a quadratic full-relation cycle need not consist of perfect arrows.
If an arrow on the cycle has more than one zero continuation on one side, the uniqueness condition for a perfect pair fails.
Thus such a cycle may contribute no new vertex to \(Q^\Gamma\); see Example~\ref{ex:nonperfect-full-cycle}.

We now describe the algebraic data used to construct the surface in Corollary~\ref{cor:End-string-G}.
Let
\[
\Gamma=A^{\mathrm{CMA}}\cong\mathbf{k}Q^\Gamma/J,
\]
where \((Q^\Gamma,J)\) is the bound quiver constructed in Proposition~\ref{prop:perfect-arrow-splitting}, and fix sign functions \(\sigma_\Gamma\) and \(\varepsilon_\Gamma\) for \((Q^\Gamma,J)\).
Applying the \(S_T\)-segment construction of Definition~\ref{def:S-segment} to \((Q^\Gamma,J)\) yields a bound quiver \((Q_T,J_T)\) such that \(\End_A(T)\cong\mathbf{k}Q_T/J_T\).

For an arrow \([u]\) of \(Q_T\), represented by an \(S_T\)-segment \(u=a_1\cdots a_r\), set \(\sigma_T([u])=\sigma_\Gamma(a_1)\) and \(\varepsilon_T([u])=\varepsilon_\Gamma(a_r)\).
For a path \(q=[u_1]\cdots[u_s]\) in \(Q_T\), define \(\operatorname{exp}(q)=u_1\cdots u_s\).
By the definition of \(J_T\), one has \(q\in J_T\) if and only if \(\operatorname{exp}(q)\in J\).

The functions \(\sigma_T\) and \(\varepsilon_T\) are sign functions for \((Q_T,J_T)\).
Indeed, if two \(S_T\)-segments with the same source had the same first arrow, condition~\ref{S2} for \((Q^\Gamma,J)\) would force them to agree until the next vertex of \(S_T\), and {therefore} they would be equal.
Thus distinct arrows of \(Q_T\) with the same source have opposite \(\sigma_T\)-values.
The corresponding statement for \(\varepsilon_T\) follows dually.

Moreover, if \([u][v]\notin J_T\), then the last arrow \(a\) of \(u\) and the first arrow \(b\) of \(v\) satisfy \(ab\notin J\).
{Therefore} \(\sigma_T([v])=-\varepsilon_T([u])\).
Consequently, every composable pair satisfying \(\sigma_T([v])=\varepsilon_T([u])\) is a quadratic minimal relation of \(J_T\).

Let \(\mathcal R_T\) be the set of minimal monomial generators of \(J_T\).
Let \(\mathcal V_T\) consist of the relations \(r\in\mathcal R_T\) such that either \(\ell(r)\geqslant3\), or \(r=[u][v]\) and \(\sigma_T([v])=-\varepsilon_T([u])\).
Define \(J_T^{\mathrm{lg}} =\langle\mathcal R_T\setminus\mathcal V_T\rangle\).

\begin{lemma}\label{lem:compressed-minimal-relations}
One has \(J_T^{\mathrm{lg}} =\langle [u][v]\mid \sigma_T([v])=\varepsilon_T([u])\rangle\), where only composable pairs are included.
Consequently, \(\mathbf{k}Q_T/J_T^{\mathrm{lg}}\) is locally gentle.
Moreover, if \(r=[u_1]\cdots[u_m]\in\mathcal V_T\), then \(r\) is nonzero modulo \(J_T^{\mathrm{lg}}\) and \(\sigma_T([u_{i+1}])=-\varepsilon_T([u_i])\) for every \(i\).
\end{lemma}

\begin{proof}
The preceding sign rule shows that every composable pair with equal endpoint signs is a 2-relation of \(J_T\).
Conversely, every relation in \(\mathcal R_T\setminus\mathcal V_T\) is quadratic and does not have opposite endpoint signs, so its endpoint signs are equal.
This proves the stated description of \(J_T^{\mathrm{lg}}\).

By Proposition~\ref{prop:corner-string-compression}, \(Q_T\) has at most two arrows starting or ending at each vertex.
For a fixed arrow, distinct successors have opposite \(\sigma_T\)-values.
Thus at most one successor has endpoint sign equal to its \(\varepsilon_T\)-value and at most one has the opposite sign.
The former composition belongs to \(J_T^{\mathrm{lg}}\), whereas the latter does not.
The dual argument applies to predecessors.
Therefore \(\mathbf{k}Q_T/J_T^{\mathrm{lg}}\) is locally gentle.

Let \(r=[u_1]\cdots[u_m]\in\mathcal V_T\).
If \(m=2\), then \(\sigma_T([u_2])=-\varepsilon_T([u_1])\) by the definition of \(\mathcal V_T\).
If \(m\geqslant3\), the minimality of \(r\) implies that every quadratic subpath \([u_i][u_{i+1}]\) is nonzero modulo \(J_T\).
The sign rule therefore {provides} \(\sigma_T([u_{i+1}])=-\varepsilon_T([u_i])\) for every \(i\).
In either case, no quadratic subpath of \(r\) belongs to \(J_T^{\mathrm{lg}}\), and {therefore} \(r\) is nonzero modulo this ideal.
\end{proof}

The sign functions determine the ordering data needed for the surface construction.
For each \(i\in(Q_T)_0\), introduce two formal endpoint incidences \(i^+\) and \(i^-\).
An arrow \(a:i\to j\) specifies that \(i^{\sigma_T(a)}\) immediately precedes \(j^{-\varepsilon_T(a)}\) in the clockwise order of a fan.
The sign conditions imply that every incidence has at most one predecessor and one successor.
{Therefore} the incidences decompose into linearly ordered and cyclic fans, with isolated incidences regarded as one-element linear fans.

Applying the surface construction of \cite[Definition~4.6]{PPP19}, linear fans determine boundary sectors, cyclic fans determine red punctures, and the complementary regions contain the white marked points.
The two incidences of a loop are kept distinct, even when they belong to the same fan.

Let \(r=a_1\cdots a_m\in\mathcal V_T\).
By Lemma~\ref{lem:compressed-minimal-relations}, the endpoint signs of each consecutive pair are opposite.
Thus \(r\) determines \(m+1\) successive incidences in one fan.
{As} \(J_T\) is admissible, every \(r\in\mathcal V_T\) has length at least two.
{Therefore} these \(m+1\) incidences form a label in the sense of Definition~\ref{def:label}.
Drawing a label along these incidences {provides} the relation of type \ref{R2} corresponding to \(r\).
For a cyclic fan, the incidences are read periodically, so repeated incidences and winding multiplicities are retained.

\begin{theorem}\label{thm:geometric-End}
Assume that \(A\) satisfies the G-condition.
For every nonzero basic \(T\in\mathsf{Gproj}(A)\), the preceding construction yields a labelled tiled surface \((\mathcal S_T,\mathsf P_T,\mathsf L_T)\) such that
\begin{equation}\label{eq:End-labelled-tiling}
A_{\mathsf P_T,\mathsf L_T}
\cong\mathbf{k}Q_T/J_T
\cong e_TA^{\mathrm{CMA}}e_T
\cong\End_A(T).
\end{equation}
More precisely:
\begin{enumerate}[label={\rm(\arabic*)}]
\item each arrow \([u]\) of \(Q_T\) is represented by a nonzero path \(u\) in \(Q^\Gamma\) whose endpoints lie in \(S_T\) and whose internal vertices lie outside \(S_T\);

\item a path \(q\) in \(Q_T\) is zero modulo \(J_T\) if and only if its expansion \(\operatorname{exp}(q)\) is zero in \(A^{\mathrm{CMA}}\);

\item the relations in \(\mathcal R_T\setminus\mathcal V_T\) are precisely the relations of type \ref{R1}, while those in \(\mathcal V_T\) are precisely the relations of type \ref{R2}.
\end{enumerate}
\end{theorem}

\begin{proof}
By Lemma~\ref{lem:compressed-minimal-relations}, \(\mathbf{k}Q_T/J_T^{\mathrm{lg}}\) is locally gentle.
Applying the surface construction of \cite[Definition~4.6 and Theorem~4.10]{PPP19}, with the colour convention of Section~\ref{subsection: geo-module categories}, {provides} a tiled surface \((\mathcal S_T,\mathsf P_T)\).
Its dissection arcs are indexed by the vertices of \(Q_T\), and {therefore} by the indecomposable summands of \(T\).

Let \(r=a_1\cdots a_m\in\mathcal V_T\).
By Lemma~\ref{lem:compressed-minimal-relations}, \(r\) is nonzero modulo \(J_T^{\mathrm{lg}}\), and its consecutive arrows follow one fan.
It therefore determines a label consisting of \(m+1\) successive incidences.
For a cyclic fan, these incidences are read periodically.
Let \(\mathsf L_T\) be the set of labels obtained in this way.

We verify that every red puncture carries a label.
Let \(c\) be the nonzero oriented cycle determined by one period of the corresponding cyclic fan.
{As} \(\mathbf{k}Q_T/J_T\cong e_TA^{\mathrm{CMA}}e_T\) is finite-dimensional, \(c^N\in J_T\) for some \(N\geqslant1\).
Because every power of \(c\) is nonzero modulo \(J_T^{\mathrm{lg}}\), a minimal generator of \(J_T\) occurring in \(c^N\) cannot belong to \(\mathcal R_T\setminus\mathcal V_T\).
It therefore belongs to \(\mathcal V_T\) and {provides} a label at this puncture.

The relations of type \ref{R1} generate \(J_T^{\mathrm{lg}}\), while the labels in \(\mathsf L_T\) {provide} exactly the relations in \(\mathcal V_T\) of type \ref{R2}.
{Therefore} $I_{\mathsf P_T,\mathsf L_T} =J_T^{\mathrm{lg}}+\langle\mathcal V_T\rangle =\langle\mathcal R_T\rangle =J_T.$ Therefore \(A_{\mathsf P_T,\mathsf L_T}\cong\mathbf{k}Q_T/J_T\).
Proposition~\ref{prop:corner-string-compression} {provides} \(\mathbf{k}Q_T/J_T\cong e_TA^{\mathrm{CMA}}e_T\), and Lemma~\ref{prop:End-corner} {provides} \(e_TA^{\mathrm{CMA}}e_T\cong\End_A(T)\).

Assertion~(1) follows from the \(S_T\)-segment construction, assertion~(2) from the definition of \(J_T\), and assertion~(3) from Lemma~\ref{lem:compressed-minimal-relations}.
\end{proof}

Let \(\mathcal D\) be a Gorenstein-projective support collection, and let \((T_{\mathcal D},P_{\mathcal D})\) be the corresponding support \(\tau\)-tilting pair under Proposition~\ref{thm:GP-support-tau-geometry}.
The indecomposable summands of \(T_{\mathcal D}\) are represented by the permissible clockwise-most \(\rpoint\)-arcs in \(\mathcal D\), whereas the arcs in \(\mathcal D\cap\mathsf P\) represent the summands of \(P_{\mathcal D}\).
Thus Theorem~\ref{thm:geometric-End} applies to \(T_{\mathcal D}\) whenever \(T_{\mathcal D}\neq0\); the projective component \(P_{\mathcal D}\) does not contribute additional summands to \(T_{\mathcal D}\).

\begin{remark} \rm \label{rem:two-surfaces-End}
Two geometric descriptions play different roles.
On the fixed labelled surface \((\mathcal S,\mathsf P,\mathsf L)\) of \(A\), Theorem~\ref{thm:homs} and the geometric interpretation following Theorem~\ref{thm:geometric-hom} {provide} a geometric graph-map basis of each \(\Hom_A(T_j,T_i)\).
Together with Lemma~\ref{thm:bound-quiver-End}, these maps determine the arrow space \(\rad_A(T_j,T_i)/\rad_{\add T}^{2}(T_j,T_i)\).
Choosing a basis of this quotient {provides} the arrows \(i\to j\), and the relations are obtained by composing their representatives.

By contrast, Theorem~\ref{thm:geometric-End} constructs a labelled tiled surface \((\mathcal S_T,\mathsf P_T,\mathsf L_T)\) whose labelled tiling algebra is isomorphic to \(\End_A(T)\).
This surface depends on the chosen sign functions and need not be a subsurface of \(\mathcal S\).
Only the isomorphism class of the algebra in \eqref{eq:End-labelled-tiling} is canonical; the topology of the surface and its red punctures may depend on these choices \cite{XZ26}.
\end{remark}

\begin{corollary}\label{cor:uniform-string-condition}
For a string algebra \(A\), the following are equivalent:
\begin{enumerate}[label={\rm(\arabic*)}]
\item \(A\) satisfies the G-condition; \item \(\End_A(T)\) is a string algebra for every nonzero basic Gorenstein-projective \(A\)-module \(T\).
\end{enumerate}
\end{corollary}

\begin{proof}
The implication \((1)\Rightarrow(2)\) is Corollary~\ref{cor:End-string-G}.
Conversely, take \(T=G_A\).
Then \(\End_A(G_A)=A^{\mathrm{CMA}}\), so Lemma~\ref{thm:CMA-string-G} {provides} the G-condition.
\end{proof}

Even without the G-condition, if \(e_TA^{\mathrm{CMA}}e_T\cong\End_A(T)\) is a string algebra for a fixed \(T\), then Theorem~\ref{thm:string-labelled-tiling} provides a labelled tiled-surface model for this particular endomorphism algebra.
Thus Corollary~\ref{cor:uniform-string-condition} characterizes the uniform assertion for all Gorenstein-projective modules.
No analogous converse is claimed when \(T\) is restricted to Gorenstein-projective support \(\tau\)-tilting modules.
Example~\ref{ex:tau-tilting-nonmonomial} shows only that, outside the G-condition, modules in this restricted class may have endomorphism algebras which are not string algebras.

Let \((\mathcal S_\Gamma,\mathsf P_\Gamma,\mathsf L_\Gamma)\) be the labelled tiled surface associated with \(\Gamma=A^{\mathrm{CMA}}\) and the chosen sign functions.
Every nonzero direct path in \(\Gamma\) follows successive angles in one fan.
{Therefore} an \(S_T\)-segment \(u\) corresponds to a fan interval whose internal dissection arcs are indexed by the summands of \(G_A\) not contained in \(T\).
On \((\mathcal S_T,\mathsf P_T,\mathsf L_T)\), the compressed arrow \([u]\) replaces this interval by a single angle.
This is the geometric meaning of path compression.

Under the G-condition, Proposition~\ref{prop:band-criterion} and Theorem~\ref{thm:geometric_modules} {provide}
\begin{equation}\label{eq:End-band-test}
\begin{aligned}
\End_A(T)\text{ is representation-finite}
\quad\Longleftrightarrow\quad&
(\mathcal S_T,\mathsf P_T,\mathsf L_T)
\text{ has no permissible closed curve}\\
&\text{meeting }\mathsf P_T\text{ at least twice}.
\end{aligned}
\end{equation}

\subsection{Stable endomorphism algebras}

For \(A\)-modules \(X\) and \(Y\), write
\[ \underline{\Hom}_A(X,Y) =\Hom_A(X,Y)/\mathcal P_A(X,Y),
   \quad \underline{\End}_A(X) =\underline{\Hom}_A(X,X), \]
where \(\mathcal P_A(X,Y)\) consists of the morphisms which factor through projective \(A\)-modules.

Let \(M\) be a basic \(\tau\)-rigid module over a string algebra \(A\), and write \(M=M_{\mathrm{np}}\oplus P\), where \(P\) is projective and \(M_{\mathrm{np}}\) has no projective direct summand.
{Auslander--Reiten duality \cite[Chapter~IV]{ARS95} identifies $D\Ext_A^1(M,M)$ with a quotient of $\Hom_A(M,\tau M)$.}
{The $\tau$-rigidity of $M$ therefore implies $\Ext_A^1(M,M)=0$.}
{As} every string algebra is special biserial, \cite[Theorem~1.1]{SZ03} implies that \(\underline{\End}_A(M_{\mathrm{np}})\) is gentle whenever \(M_{\mathrm{np}}\neq0\).
{As} projective modules are zero objects in the stable category, $\underline{\End}_A(M) \cong\underline{\End}_A(M_{\mathrm{np}})$.
If \(M_{\mathrm{np}}=0\), this algebra is zero.
Under the G-condition, the stable endomorphism algebras of Gorenstein-projective modules admit the following more precise description.

The following consequence can also be deduced from \cite[Proposition~5.9]{CSZ18}.
We include a direct proof using the corner presentation.

\begin{corollary}\label{prop:stable-End-semisimple-G}
Assume the G-condition.
Let \(T\in\mathsf{Gproj}(A)\) be basic and have exactly \(r\) non-projective indecomposable direct summands.
Then
\[
\underline{\End}_A(T)\cong\mathbf{k}^{\,r},
\]
where \(\mathbf{k}^{\,0}\) denotes the zero algebra.
No \(\tau\)-rigidity hypothesis is required.
\end{corollary}

\begin{proof}
{Let} \(\Gamma=\End_A(G_A)\), and let \(e_0\) be the sum of the primitive idempotents corresponding to the projective modules \(e_iA\).
Under the isomorphism \(\End_A(T)\cong e_T\Gamma e_T\) of Lemma~\ref{prop:End-corner}, the ideal of endomorphisms of \(T\) which factor through projective modules corresponds to $\mathcal P_A(T,T)=e_T\Gamma e_0\Gamma e_T$.
Indeed, every finitely generated projective \(A\)-module belongs to \(\add A\), so this ideal includes factorizations through projective modules not contained in \(\add T\).

In the bound quiver of Proposition~\ref{prop:perfect-arrow-splitting}, every arrow has at least one endpoint at an original vertex \(i\in Q_0\).
Consequently, quotienting by \(\Gamma e_0\Gamma\) removes all original vertices and all arrows, leaving only the vertices \(v_\alpha\), \(\alpha\in{\perpath(A)}\).
{Therefore}
\[
\Gamma/\Gamma e_0\Gamma \cong \prod_{\alpha\in{\perpath(A)}}\mathbf{k}.
\]
It follows that
\[
\underline{\End}_A(T) \cong \frac{e_T\Gamma e_T} {e_T\Gamma e_0\Gamma e_T} \cong e_T\bigl(\Gamma/\Gamma e_0\Gamma\bigr)e_T \cong \mathbf{k}^{\,r}.
\]
\end{proof}

Taking \(T=G_A\) {provides} $\underline{\End}_A(G_A) \cong \mathbf{k}^{\,|{\perpath(A)}|}$.
{Here $\underline{\mathsf{Gproj}}(A)$ denotes the quotient of $\mathsf{Gproj}(A)$ by the ideal of morphisms which factor through projective modules.}
{The image of $G_A$ is an additive generator of this stable category, and the usual additive-generator equivalence identifies it with the category of finitely generated projective modules over $\underline{\End}_A(G_A)$, up to the harmless opposite-algebra convention determined by sidedness.}
{Since this stable endomorphism algebra is semisimple, its category of finitely presented functors is semisimple.}
{Thus \(A\) is an \(\Omega_{\mathcal G}\)-algebra in the terminology of \cite{H21}.}
This is consistent with the semisimplicity criterion in \cite[Proposition~5.9]{CSZ18} and with the framework of \(\Omega_{\mathcal G}\)-algebras in \cite{H21}.

\section{Representation-finiteness}

We apply the preceding constructions to representation-finiteness, regarding the zero algebra as representation-finite.

\begin{lemma}\label{lem:repfinite-corner}
Let \(\Gamma\) be a finite-dimensional representation-finite algebra and let \(e\in\Gamma\) be an idempotent.
Then \(e\Gamma e\) is representation-finite.
\end{lemma}

\begin{proof}
Set \(B=e\Gamma e\).
For right modules, the functor
\[
F=-\otimes_B e\Gamma: \operatorname{mod}B\longrightarrow\operatorname{mod}\Gamma
\]
is left adjoint to \(G=(-)e\).
For every \(X\in\operatorname{mod}B\), one has
\[
(FX)e \cong X\otimes_B e\Gamma e \cong X.
\]
{Therefore} the unit of the adjunction is an isomorphism, and \(F\) is fully faithful.
In particular,
\[
\End_B(X)\cong\End_\Gamma(FX).
\]
Thus \(F\) preserves indecomposability and reflects isomorphisms.
It therefore induces an injection from the isomorphism classes of indecomposable \(B\)-modules to those of indecomposable \(\Gamma\)-modules.
{As} \(\Gamma\) is representation-finite, so is \(B=e\Gamma e\).
\end{proof}

The next lemma compares the bands of \(A\) and \(A^{\mathrm{CMA}}\) using the bound quiver constructed in Proposition~\ref{prop:perfect-arrow-splitting}.

\begin{lemma}\label{lem:band-splitting-contraction}
Assume that \(A\) satisfies the G-condition, and {let} \(\Gamma=A^{\mathrm{CMA}}\).
Let \(w\) be a band of \(A\).
Replacing every occurrence of a perfect arrow \(\alpha\) in \(w\) by \(\alpha^-\alpha^+\), and every occurrence of \(\alpha^{-1}\) by \((\alpha^+)^{-1}(\alpha^-)^{-1}\), produces a band \(w^*\) of \(\Gamma\).

Conversely, every band of \(\Gamma\) is obtained in this way: contracting each pair \(\alpha^-\alpha^+\), or its inverse, to \(\alpha\), or \(\alpha^{-1}\), produces a band of \(A\).
These operations induce mutually inverse bijections between the band classes of \(A\) and those of \(\Gamma\).
\end{lemma}

\begin{proof}
Set \(R={\perpath(A)}\), and use the bound quiver \(\Gamma\cong\mathbf{k}Q^\Gamma/J\) constructed in Proposition~\ref{prop:perfect-arrow-splitting}.
For a letter \(c\) in a string of \(A\), define
\[
c^*=
\begin{cases}
\alpha^-\alpha^+,
& c=\alpha\in R,\\
(\alpha^+)^{-1}(\alpha^-)^{-1},
& c=\alpha^{-1}\text{ with }\alpha\in R,\\
c,
& \text{otherwise}.
\end{cases}
\]
For a word \(w=c_1\cdots c_s\), {let} \(w^*=c_1^*\cdots c_s^*\).

Each \(c^*\) has the same endpoints as \(c\).
Moreover, two consecutive expanded blocks cancel precisely when the corresponding letters of \(w\) cancel.
{Therefore} \(w\) is reduced if and only if \(w^*\) is reduced.

By Proposition~\ref{prop:perfect-arrow-splitting}, a perfect-pair relation \(\alpha\beta\), with \(\alpha,\beta\in R\), is replaced by \(\alpha^+\beta^-\), whereas every other minimal relation is replaced by its full expansion.
The same holds for the inverse words.
Consequently, for every \(n\geqslant1\), the word \(w^n\) is a string over \(A\) if and only if \((w^*)^n\) is a string over \(\Gamma\).
Thus expansion sends every cyclic string whose positive powers are strings to a cyclic string with the same property.

Conversely, let \(z\) be such a cyclic string over \(\Gamma\).
The new vertex \(v_\alpha\) is incident only with \(\alpha^-\) and \(\alpha^+\).
{As} \(z\) is reduced, every visit to \(v_\alpha\) occurs through \(\alpha^-\alpha^+\) or \((\alpha^+)^{-1}(\alpha^-)^{-1}\).
Contract these subwords to \(\alpha\) and \(\alpha^{-1}\), respectively.
The resulting cyclic word \(\overline z\) is reduced, and the preceding comparison of relations shows that all its positive powers are strings over \(A\).
By construction, \((\overline z)^*=z\) up to cyclic permutation, and \(\overline{w^*}=w\).
{Therefore} expansion and contraction are inverse on these cyclic strings.

It remains to check primitivity.
Expansion preserves proper powers because \((u^m)^*=(u^*)^m\).
Conversely, suppose that \(w^*=u^m\) for some \(m\geqslant2\).
{As} no arrow of \(Q^\Gamma\) joins two new vertices, the cyclic word \(u\) passes through an original vertex.
After a cyclic permutation, we may assume that \(u\) begins at such a vertex.
It can then be contracted without cutting a pair \(\alpha^-\alpha^+\), and we obtain \(w=(\overline u)^m\) up to cyclic permutation.
Thus \(w^*\) is a proper power if and only if \(w\) is a proper power.

Both expansion and contraction commute with cyclic permutation and inversion.
Therefore they induce the stated mutually inverse bijections between band classes.
\end{proof}

\begin{theorem}\label{thm:rep-finite-G-condition}
Let \(A\) be a string algebra satisfying the G-condition.
Then the following statements are equivalent:
\begin{enumerate}[label={\rm(\arabic*)}]
\item \(A\) is representation-finite; \item \(A^{\mathrm{CMA}}\) is representation-finite; \item \(\End_A(T)\) is representation-finite for every basic Gorenstein-projective \(A\)-module \(T\); \item \(\End_A(T)\) is representation-finite for every basic Gorenstein-projective support \(\tau\)-tilting \(A\)-module \(T\).
\end{enumerate}
\end{theorem}

\begin{proof}
Set \(\Gamma=A^{\mathrm{CMA}}\).
By Lemma~\ref{thm:CMA-string-G}, \(\Gamma\) is a string algebra.
Proposition~\ref{prop:band-criterion} shows that a string algebra is representation-finite if and only if it has no bands.
By Lemma~\ref{lem:band-splitting-contraction}, \(A\) and \(\Gamma\) have the same bands.
{Therefore} \((1)\Longleftrightarrow(2)\); this also recovers \cite[Theorem~4.3]{LZ24}.

Assume (2).
The assertion is immediate for \(T=0\).
If \(T\neq0\), Lemma~\ref{prop:End-corner} {provides} \(\End_A(T)\cong e_T\Gamma e_T\).
{As} every corner of a representation-finite algebra is representation-finite by Lemma~\ref{lem:repfinite-corner}, condition~(3) follows.
The implication \((3)\Longrightarrow(4)\) is immediate.

Finally, assume (4).
The right regular module \(A_A\) is basic, projective and \(\tau\)-tilting.
In particular, it is a Gorenstein-projective support \(\tau\)-tilting module.
{As} \(\End_A(A_A)\cong A\), condition (4) implies (1).
\end{proof}

Combining Theorem~\ref{thm:rep-finite-G-condition} with \eqref{eq:End-band-test} {provides} the following geometric reformulation.
Under the G-condition, \(A\) is representation-finite if and only if, for every nonzero basic Gorenstein-projective support \(\tau\)-tilting module \(T\), the labelled surface associated with \(T\) has no permissible closed curve meeting its dissection at least twice.
For the reverse implication, take \(T=A_A\).

Without the G-condition, only partial conclusions are available.
If \(A\) is representation-finite and \(T\) is a basic projective module, then \(T\cong eA\) for some idempotent \(e\), and {therefore} \(\End_A(T)\cong eAe\) is representation-finite by Lemma~\ref{lem:repfinite-corner}.
More generally, if \(A^{\mathrm{CMA}}\) is representation-finite, then Lemma~\ref{prop:End-corner} and Lemma~\ref{lem:repfinite-corner} show that \(\End_A(T)\) is representation-finite for every basic Gorenstein-projective module \(T\).
In particular, this covers the CM-free case.

However, a representation-finite string algebra may have a representation-infinite Cohen--Macaulay Auslander algebra \cite[Remark~4.4]{LZ24}.
This does not answer the question for support \(\tau\)-tilting modules, because the module \(G_A\) used to define \(A^{\mathrm{CMA}}\) is not required to be support \(\tau\)-tilting.
This leaves the following problem.

\begin{problem}\label{prob:unrestricted}
Let \(A\) be a representation-finite string algebra.
Is \(\End_A(T)\) representation-finite for every basic Gorenstein-projective support \(\tau\)-tilting \(A\)-module \(T\)?
\end{problem}

The stable quotient does not settle this problem.
Even if \(\underline{\End}_A(T)\) is representation-finite, representation-finiteness need not lift from this quotient to \(\End_A(T)\).
One must also control the ideal \(\mathcal P_A(T,T)\) of maps factoring through projective modules, including projectives not belonging to \(\add T\).
Moreover, Example~\ref{ex:tau-tilting-nonmonomial} shows that, outside the G-condition, \(\End_A(T)\) need not be a string algebra even when \(T\) is Gorenstein-projective and \(\tau\)-tilting.
{Therefore} the band criterion is not available in general.

\section{Examples}

{We conclude with examples illustrating the preceding results.}

\begin{example}\rm \label{ex:nongentle-G}
Let $A=\mathbf{k}Q/I$, where
        \[Q=\xymatrix{5 \ar[r]^x & 4 \ar[r]^y & 3 \ar[r]^z
          & 2 \ar@/^0.5pc/[r]^{\beta}
          & 1 \ar@/^0.5pc/[l]^{\alpha}},
          \quad I=\langle \alpha\beta, \beta\alpha, xyz\rangle\]
%\[
%Q=\begin{tikzcd}[column sep=2.8em,ampersand replacement=\&]
%5\arrow[r,"x"] \& 4\arrow[r,"y"] \& 3\arrow[r,"z"] \&
%2\arrow[r,bend left=18,"\beta"] \& 1\arrow[l,bend left=18,"\alpha"]
%\end{tikzcd},
%\quad
%I=\langle\alpha\beta,\beta\alpha,xyz\rangle.
%\]
The algebra $A$ is not gentle because $xyz$ is a minimal relation of length three.
Its only relation cycle supporting a perfect path sequence is the gentle cycle determined by $\alpha$ and $\beta$.
{Therefore} $A$ satisfies the G-condition, and
\[
\perpath(A)=\{\alpha,\beta\}, \quad \alpha A\cong S_2, \quad \beta A\cong S_1.
\]
There is no band: the tail cannot occur in a reduced cyclic string, while both orientations of the two-cycle, and their inverses, contain a relation.
Thus $A$ is representation-finite.

{Let} $P_i=e_iA$ and
\[
T=S_2\oplus P_2\oplus P_3\oplus P_4\oplus P_5.
\]
The perfect successor of $\alpha$ is $\beta$, and right multiplication by $\beta$ {provides} a surjection $Te_2\to Te_1$: every nonzero path in a summand $P_i$, $i\neq1$, ending at $1$ ends in $\beta$, while $S_2e_1=0$.
Proposition~\ref{prop:perfect-successor-tau} therefore {provides} $\Hom_A(T,\tau S_2)=0$.
{As} the other summands are projective, $T$ is $\tau$-rigid.
It has five pairwise non-isomorphic indecomposable summands, so it is a Gorenstein-projective $\tau$-tilting module.
The corresponding five vertices form a clique of size $|Q_0|$ in $\mathcal G_A$, and its maximal good collection is
\[
\{\gamma_{S_2}^{\circlearrowright}, \gamma_{P_2}^{\circlearrowright},\ldots, \gamma_{P_5}^{\circlearrowright}\}.
\]

Figure~\ref{fig:example-6-1-surface-and-T} shows the labelled
tiled surface associated with $A$ and the permissible
clockwise-most arcs representing the indecomposable summands of $T$.

\begin{figure}[htbp]
	\centering
	\begin{tikzpicture}[scale=.48,>=stealth]
		\foreach \s in {-8,8} {
			\begin{scope}[shift={(\s,0)}]
				\coordinate (p) at (180:6);
				\coordinate (a) at (120:6);
				\coordinate (b) at (60:6);
				\coordinate (c) at (0:6);
				\coordinate (q) at (-60:6);
				\coordinate (w5) at (150:6);
				\coordinate (w4) at (90:6);
				\coordinate (w3) at (30:6);
				\coordinate (w2) at (-30:6);
				\coordinate (w1) at (-120:6);
				\coordinate (r) at (0,-4.2);
				\draw[black,line width=1.5pt] (0,0) circle (6);
				\draw[red,line width=1pt] (p)--(a);
				\draw[red,line width=1pt] (p)--(b);
				\draw[red,line width=1pt] (p)--(c);
				\draw[red,line width=1pt] (p)--(q);
				\draw[red,line width=1pt] (p)
				.. controls (-4.8,-4.8) and (0,-6) .. (q);
				
				% The label crosses 5,4,3,2, reading xyz.
				\begin{scope}[shift={(p)}]
					\draw[blue,line width=1.2pt,->]
					(60:2.25) arc[start angle=60,end angle=-30,radius=2.25];
				\end{scope}
				
				\ifnum\s<0
				% Vertex labels on the dissection arcs.
				\node[red,font=\scriptsize] at (-3.4,3.5) {$x_5$};
				\node[red,font=\scriptsize] at (0,3.95) {$x_4$};
				\node[red,font=\scriptsize] at (3.1,.38) {$x_3$};
				\node[red,font=\scriptsize] at (-.2,-2.92) {$x_2$};
				\node[red,font=\scriptsize] at (-2,-4.5) {$x_1$};
				\node[blue,font=\scriptsize] at (-3.15,1.0) {$xyz$};
				\begin{scope}[shift={(p)}]
					\foreach \u/\v in {60/30,30/0,0/-30}
					\draw[dark-green,line width=1.2pt,->]
					(\u:1.12) arc[start angle=\u,end angle=\v,radius=1.12];
					\draw[dark-green,line width=1.2pt,->]
					(-30:1.15) arc[start angle=-30,end angle=-73,radius=1.15];
				\end{scope}
				\node[dark-green,font=\scriptsize] at (-4.85,1.45) {$x$};
				\node[dark-green,font=\scriptsize] at (-4.35,.52) {$y$};
				\node[dark-green,font=\scriptsize] at (-4.38,-.52) {$z$};
				\node[dark-green,font=\scriptsize] at (-4.7,-1.42) {$\beta$};
				\begin{scope}[shift={(q)}]
					\draw[dark-green,line width=1.2pt,->]
					(195:1.0) arc[start angle=195,end angle=150,radius=1.0];
				\end{scope}
				\node[dark-green,font=\scriptsize] at (1.5,-4.75) {$\alpha$};
				
				\else
				% White-circle representatives; the crossing words are checked.
				% P5: 5,4,3 (xy). ?? ???????
				\draw[violet,line width=1.2pt] (w5)
				to (w2);
				% P4: 4,3,2,1 (yz beta). ?? ??
				\draw[purple,line width=1.2pt] (w4) to (w1);
				% P3: 3,2,1 (z beta). ?? ??
				\draw[purple!75,line width=1.2pt, bend right=10] (w3) to (w1);
				% P2: 2,1 (beta). ?? ?????????
				\draw[purple!55,line width=1.2pt, bend right=20] (w2) to (w1);
				% S2: 2 (the trivial string at 2). ?? ??
				\draw[teal!75!black,line width=1.2pt, bend right=15] (r) to (w2);
				\node[violet,font=\scriptsize] at (-3.85,2.73) {$P_5$};
				\node[purple,font=\scriptsize] at (.2,4.7) {$P_4$};
				\node[purple!75,font=\scriptsize] at (3.55,2.7) {$P_3$};
				\node[purple!55,font=\scriptsize] at (2,-2.5) {$P_2$};
				\node[teal!75!black,font=\scriptsize] at (2.2,-3.65) {$S_2$};
				
				\fi
				
				% Draw all marked points last.
				\foreach \u in {p,a,b,c,q}
				\draw[red,fill=red,thick] (\u) circle (.16);
				\foreach \u in {w5,w4,w3,w2,w1,r}
				\draw[black,fill=white,thick] (\u) circle (.16);
			\end{scope}
		}
	\end{tikzpicture}
	\caption{The labelled tiled surface of $A$ in Example \ref{ex:nongentle-G} (left)
		and the permissible $\circ$-arcs representing
		$T=S_2\oplus P_2\oplus P_3\oplus P_4\oplus P_5$ (right).}
	\label{fig:example-6-1-surface-and-T}
\end{figure}
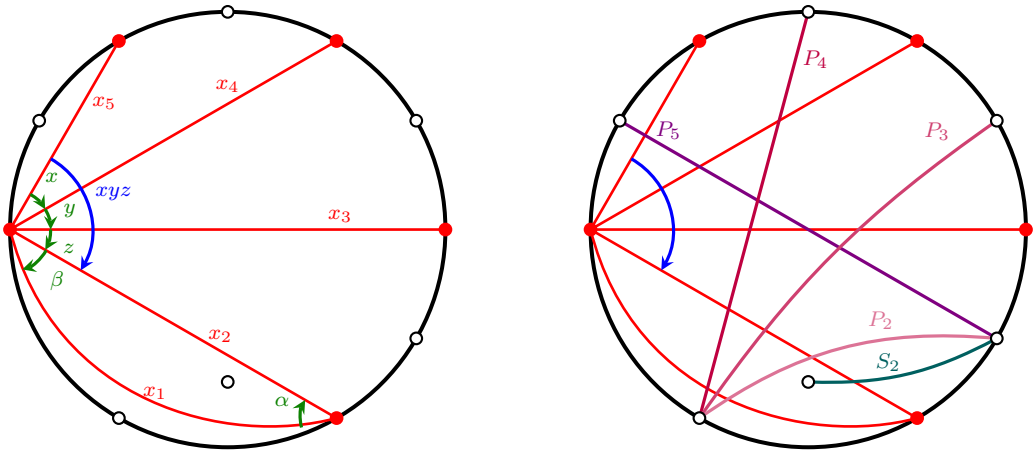

The Cohen--Macaulay Auslander quiver is obtained by splitting $\alpha$ and $\beta$ through the vertices $v_\alpha$ and $v_\beta$.
Its relations are
\[
xyz,\quad \alpha^+\beta^-,\quad \beta^+\alpha^-;
\]
For the corner determined by $T$, the selected vertices are $S_T=\{v_\alpha,2,3,4,5\}$.
The path from $2$ through the omitted vertices $v_\beta,1$ to $v_\alpha$ contains $\beta^+\alpha^-$ and is zero.
{Therefore} it {provides} no compressed arrow.
The surviving segments are $x,y,z,\alpha^+$, and
\begin{equation}\label{eq:example-long-label-End}
\End_A(T)\cong
\mathbf{k}\left(
5\xrightarrow{x}4\xrightarrow{y}3\xrightarrow{z}2
\xleftarrow{h}v_\alpha\right)/\langle xyz\rangle,
\quad h=[\alpha^+].
\end{equation}
Here $h$ represents the quotient $P_2\twoheadrightarrow S_2$.
A path basis consists of the five vertex idempotents, the four arrows $x,y,z,h$, and the two paths $xy,yz$.
Thus $\dim_{\mathbf{k}}\End_A(T)=11$.
The quiver in \eqref{eq:example-long-label-End} has no band, so $\End_A(T)$ is representation-finite, in accordance with Theorem~\ref{thm:rep-finite-G-condition}.
Moreover, $T$ has exactly one non-projective indecomposable summand, and {therefore}
\[
\underline{\End}_A(T)\cong\mathbf{k}
\]
by Proposition~\ref{prop:stable-End-semisimple-G}.

Choose $\sigma_T=1$ on every arrow, $\varepsilon_T(x)=\varepsilon_T(y)=\varepsilon_T(z)=-1$, and $\varepsilon_T(h)=1$.
These signs are inherited from the Cohen--Macaulay Auslander quiver by taking all $\sigma_\Gamma=1$, $\varepsilon_\Gamma(\alpha^+)=\varepsilon_\Gamma(\beta^+)=1$, and $\varepsilon_\Gamma=-1$ on the remaining arrows.
Consequently,
\[
J_T^{\mathrm{lg}}=0, \quad \mathcal V_T=\{xyz\}.
\]
Thus the relation $xyz$ is imposed by a label.
The corresponding fan interval is $(5,4,3,2)$, while the angle from $v_\alpha$ to $2$ {provides} the arrow $h$.
Figure~\ref{fig:example-End-surface} displays the resulting labelled tiled surface.

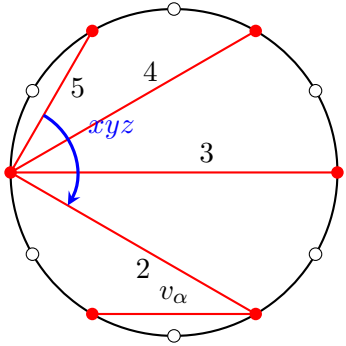
\begin{figure}[htbp]
\centering
\begin{tikzpicture}[scale=0.94,>=stealth]
\draw[thick] (0,0) circle (2.3);
\coordinate (p) at (180:2.3);
\coordinate (a) at (120:2.3);
\coordinate (b) at (60:2.3);
\coordinate (c) at (0:2.3);
\coordinate (q) at (-60:2.3);
\coordinate (v) at (-120:2.3);
\draw[red,thick] (p)--(a) node[pos=0.60,right,black] {$5$};
\draw[red,thick] (p)--(b) node[pos=0.57,above,black] {$4$};
\draw[red,thick] (p)--(c) node[pos=0.60,above,black] {$3$};
\draw[red,thick] (p)--(q) node[pos=0.54,below,black] {$2$};
\draw[red,thick] (q)--(v) node[midway,above,black] {$v_\alpha$};
\begin{scope}[shift={(p)}]
                    \draw[blue,very thick,->]
(60:0.93) arc[start angle=60,end angle=-30,radius=0.93];
\node[blue] at (1.42,0.60) {$xyz$};
\end{scope}
\foreach \ang in {180,120,60,0,-60,-120}
\fill[red] (\ang:2.3) circle (2.5pt);
\foreach \ang in {150,90,30,-30,-90,-150}
\draw[black,fill=white] (\ang:2.3) circle (2.5pt);
\end{tikzpicture}
\caption{The labelled tiled surface for the endomorphism algebra in \eqref{eq:example-long-label-End}.}
\label{fig:example-End-surface}
\end{figure}

For comparison, let $U=A\oplus S_2$; no $\tau$-rigidity is asserted for $U$.
In the corresponding corner only $v_\beta$ is omitted.
The nonzero segment $\beta^-\beta^+$ becomes an arrow $b:2\to1$, and the bound quiver of $\End_A(U)$ is obtained from $Q$ by splitting $\alpha$ alone.
Its relations are
\[
xyz,\quad \alpha^+b,\quad b\alpha^-.
\]
The last two relations expand to paths containing $\alpha^+\beta^-$ and $\beta^+\alpha^-$, respectively.
This {provides} a nontrivial instance of path compression through an omitted vertex.
\end{example}

\begin{remark} \rm
\label{ex:nonperfect-full-cycle}
Let
\[
        Q = \xymatrix{1\ar@/^0.5pc/[r]^a & 2 \ar@/^0.5pc/[l]^b \ar[r]^c & 3},
        \quad I=\langle ab, ac , {ba}\rangle
        \]
%        \[
%Q=\begin{tikzcd}[column sep=3em,ampersand replacement=\&]
%1\arrow[r,bend left=18,"a"] \&
%2\arrow[l,bend left=18,"b"]\arrow[r,"c"] \&3
%\end{tikzcd},
%\quad
%I=\langle ab,ac,ba\rangle.
%\]
Although $ab$ is a full-relation oriented cycle, $(a,b)$ is not a perfect pair because both $ab$ and $ac$ are zero.
The pair $(b,a)$ is perfect, but $a$ has no perfect successor, so it belongs to no perfect path sequence.
{Therefore} $\perpath(A)=\varnothing$, the G-condition holds vacuously, and $A^{\mathrm{CMA}}\cong A$.

For $T=A$, choose
\[
\begin{array}{c|rrr}
&a&b&c\\ \hline
\sigma&1&1&-1\\
\varepsilon&1&1&-1
\end{array}
\]
as sign functions.
Then
\[
J_T^{\mathrm{lg}}=\langle ab,ba\rangle, \quad \mathcal V_T=\{ac\}.
\]
Thus $ac$ is a quadratic relation of type \textup{(R2)}, whereas $ab$ and $ba$ are relations of type \textup{(R1)}.
In particular, a quadratic relation need not be a tile relation, and an arrow must not be split merely because it lies on a {full-relation oriented cycle}.
\end{remark}

\begin{example}[A nonmonomial endomorphism algebra outside the
G-condition] \rm
\label{ex:tau-tilting-nonmonomial}
\label{ex:mixed-length-relation}
\label{ex:unrestricted-obstruction}
Fix $m\geqslant1$.
Let $Q$ be the oriented three-cycle
\begin{center}
        \begin{tikzpicture}[scale=0.54, baseline=-0.2 cm]
             \draw (2,0) node {$2,$} (-1,1.73) node{$1$} (-1,-1.73) node{$3$};
            \foreach \x in {0,120,240}
            \draw[rotate=-10+\x][->] (-1,1.73) arc(120:20:2);
            \draw[rotate=60+  0] (2.3,0) node{$x$};
            \draw[rotate=60+240] (2.3,0) node{$y$};
            \draw[rotate=60+120] (2.3,0) node{$z$};
        \end{tikzpicture}
        $\quad A_m=\mathbf{k}Q/R_Q^{\,3m+1}$.
\end{center}
%       \[
%\begin{tikzcd}[column sep=3.2em,row sep=2.4em,ampersand replacement=\&]
%1\arrow[r,"x"]\&2\arrow[d,"y"]\\
%\&3\arrow[ul,"z"]
%\end{tikzcd},
%\quad
%A_m=\mathbf{k}Q/R_Q^{\,3m+1}.
%\]
The algebra $A_m$ is a self-injective Nakayama string algebra, so all its modules are Gorenstein-projective.
Every nonzero path of positive length is perfect: the successor of a path of length $d$ is the unique path of length $3m+1-d$ beginning at its endpoint.
These successors form a permutation of the nonzero paths, and the prefix and suffix conditions follow from the uniqueness of oriented continuation.
{Therefore} the G-condition fails because perfect paths of length greater than one occur.

{Let} $P_i=e_iA_m$, $C=xyz$, and
\[
T_m=S_1\oplus P_1\oplus P_3.
\]
The path $p=C^m$ has perfect successor $x$ and satisfies $pA_m\cong S_1$.
Notice that $M(p)\cong P_1$, whereas $pA_m\cong M(e_1)$.

Right multiplication by $x$ {provides} a surjection $T_me_1\to T_me_2$.
It is trivial on the $S_1$-summand; on $P_1$ its images are $x,Cx,\ldots,C^{m-1}x$, and on $P_3$ they are $zx,zCx,\ldots,zC^{m-1}x$.
Proposition \ref{prop:perfect-successor-tau} {provides} $\Hom_{A_m}(T_m,\tau S_1)=0$.
{As} $P_1$ and $P_3$ are projective, $T_m$ is $\tau$-rigid.
It has three indecomposable summands and is therefore a Gorenstein-projective $\tau$-tilting module.

Consider the maps
\[
u:P_1\twoheadrightarrow S_1, \quad v:S_1\longrightarrow P_1,\quad {v(1)=C^m},
\]
\[
a:P_1\longrightarrow P_3,\quad {a(e_1)=z}, \quad b:P_3\longrightarrow P_1,\quad {b(e_3)=xy}.
\]
A path-basis calculation shows that $B_m=\End_{A_m}(T_m)$ has basis consisting of the three identities, $u,v$, and
\[
(ba)^j, (ab)^j\quad(1\leqslant j\leqslant m), \quad a(ba)^j, b(ab)^j\quad(0\leqslant j<m).
\]
In particular, $\dim_{\mathbf{k}}B_m=4m+5$.
With the usual composition of maps,
\[
uv=ub=av=0, \quad vu=(ba)^m\neq0,
\]
because both $vu$ and $(ba)^m$ send $e_1$ to $C^m$.
Consequently,
\begin{equation}\label{eq:nakayama-End-presentation}
B_m\cong
\mathbf{k}Q_B/
\langle [u][v],\ [u][b],\ [a][v],\
[v][u]-([b][a])^m\rangle,
\end{equation}
where
\[
Q_B=\xymatrix@C=4em{
S_1 \ar@/^/[r]^{[u]} &
P_1 \ar@/^/[l]^{[v]} \ar@/^/[r]^{[b]} &
P_3 \ar@/^/[l]^{[a]} }.
\]
Indeed, the binomial relation replaces $[v][u]$ by $([b][a])^m$, while the other three displayed relations eliminate every remaining composable pair involving $[u]$ or $[v]$.
Moreover,
\[
([b][a])^m[b]=[v][u][b]=0, \quad [a]([b][a])^m=[a][v][u]=0.
\]
It follows that the quotient in \eqref{eq:nakayama-End-presentation} is spanned by the paths corresponding to the $4m+5$ basis maps above.
Their images are linearly independent, proving the presentation and its admissibility.

The algebra $B_m$ admits no monomial presentation.
Fix its primitive idempotents.
{Here a monomial presentation means an admissible presentation on the Gabriel quiver of \(B_m\). The four arrow spaces in \(\rad B_m/\rad^2B_m\) are one-dimensional, so every choice of arrows has lifts of the form described below.}
{The displayed basis of $B_m$ also yields}
\begin{align*}
 & {e_{P_3}\rad^2(B_m)e_{P_1}=\operatorname{span}_{\mathbf k}\{a(ba)^j\mid 1\leqslant j<m\},} \\
 & {e_{P_1}\rad^2(B_m)e_{P_3}=\operatorname{span}_{\mathbf k}\{b(ab)^j\mid 1\leqslant j<m\}.}
\end{align*}
{The corresponding radical-square correction spaces for the arrows $u$ and $v$ are zero.}
Any arrow lifts of $u$ and $v$ are nonzero scalar multiples of these maps, whereas arbitrary lifts of $a$ and $b$ have the form
\[
a'=c_a a+\sum_{j=1}^{m-1}c_j a(ba)^j, \quad b'=c_b b+\sum_{j=1}^{m-1}d_j b(ab)^j, \quad c_ac_b\neq 0.
\]
{Therefore}
\[
(b'a')^m=(c_ac_b)^m(ba)^m\neq0,
\]
{Indeed, $(ba)^{m+1}=0$, and, after rewriting, every term containing at least one higher-radical correction is a scalar multiple of $(ba)^r$ for some $r\geqslant m+1$ and therefore vanishes.}
{Consequently, this element is proportional to $v'u'$ for any lifts $u',v'$.}
The corresponding paths are distinct, whereas distinct nonzero paths in a monomial algebra are linearly independent.
{By the standard conjugacy theorem for complete sets of primitive orthogonal idempotents, any other such set is conjugate to the chosen one, up to permutation, so this obstruction is intrinsic.}
Therefore $B_m$ is not a string algebra.

For $m=1$, the relation is homogeneous and $[v][u]=[b][a]\neq0$.
For $m>1$, the relation $[v][u]-([b][a])^m$ involves nonzero paths of lengths $2$ and $2m$.
In particular, $m=2$ shows that paths of different lengths must be considered together when computing $\ker\Phi_T$ in Lemma~\ref{thm:bound-quiver-End}.
Finally,
\[
\underline{\End}_{A_m}(T_m)\cong\mathbf{k},
\]
because the projective summands vanish in the stable category and the identity of $S_1$ does not factor through a projective module.
Thus the ordinary endomorphism algebra can be nonmonomial even when its stable quotient is semisimple.
{As} $A_m$ is a Nakayama algebra, it is representation-finite; the nonmonomiality of $B_m$ alone does not decide its representation type and therefore does not settle Problem~\ref{prob:unrestricted}.
\end{example}

\paragraph{Competing Interests}
The authors declare that they have no conflicts of interest as defined by the journal, nor any other interests that could be perceived as influencing the results presented in this paper.

\paragraph{Authors' Contributions}
The order of authors is alphabetical, and all authors contributed equally to the conception, methodology, derivation, and writing of this paper.

\paragraph{Ethical Approval}
This article does not require ethical approval.

\paragraph{Acknowledgements}
We would like to express our gratitude to the Guizhou Provincial Key Laboratory of Applied Mathematics and Computing Power \& Algorithms for providing us with an academic discussion laboratory.

%=========================================================

%\addcontentsline{toc}{section}{References}

% 按文章中出现顺序排列参考文献：
%   \bibliographystyle{elsarticle-num}  % 作者姓全拼、名缩写, 数字标签, 书本名正体
%   \bibliographystyle{unsrt} % 作者姓名全拼, 数字标签, 书本名斜体
% 按字母顺序排列参考文献：
%  \bibliographystyle{plain} % 作者姓名全拼, 数字标签, 书本名正体
%  \bibliographystyle{abbrv} % 作者姓全拼、名缩写, 数字标签, 书本名斜体
%  \bibliographystyle{siam}  % 作者姓全拼、名缩写 (姓名使用textsc格式), 数字标签, 书本名斜体
%  \bibliographystyle{alpha} % 作者姓名全拼, 作者标签, 书本名斜体
%  \bibliographystyle{apalike}  % 作者姓全拼、名缩写, 美国心理学学会期刊样式

% \bibliography{referLiu}

%\nocite{LZhou24}
\bibliographystyle{alpha}
\bibliography{Endomorphism_Algebras}

@article{AIR14,
  author  = {Adachi, Takahide and Iyama, Osamu and Reiten, Idun},
  title   = {{$\tau$}-tilting theory},
  journal = {Compositio Mathematica},
  volume  = {150},
  number  = {3},
  pages   = {415--452},
  year    = {2014},
  doi     = {10.1112/S0010437X13007422},
  note    = {\url{https://doi.org/10.1112/S0010437X13007422}}
}

@article{AS87,
  author  = {Assem, Ibrahim and Skowro{\'n}ski, Andrzej},
  title   = {Iterated tilted algebras of type {$\widetilde{A}_n$}},
  journal = {Mathematische Zeitschrift},
  volume  = {195},
  number  = {2},
  pages   = {269--290},
  year    = {1987},
  doi     = {10.1007/BF01166463},
  note    = {\url{https://doi.org/10.1007/BF01166463}}
}

@book{ARS95,
  author    = {Auslander, Maurice and Reiten, Idun and Smal{\o}, Sverre O.},
  title     = {Representation Theory of Artin Algebras},
  series    = {Cambridge Studies in Advanced Mathematics},
  volume    = {36},
  publisher = {Cambridge University Press},
  address   = {Cambridge},
  year      = {1995},
  doi       = {10.1017/CBO9780511623608},
  note      = {\url{https://doi.org/10.1017/CBO9780511623608}}
}

@article{BC21,
  author  = {Baur, Karin and Coelho Sim{\~o}es, Raquel},
  title   = {A geometric model for the module category of a gentle algebra},
  journal = {International Mathematics Research Notices. IMRN},
  volume  = {2021},
  number  = {15},
  pages   = {11357--11392},
  year    = {2021},
  doi     = {10.1093/imrn/rnz150},
  note    = {\url{https://doi.org/10.1093/imrn/rnz150}}
}

@article{BC21-corr,
  author  = {Baur, Karin and Coelho Sim{\~o}es, Raquel},
  title   = {Corrigendum: A geometric model for the module category of a gentle algebra},
  journal = {International Mathematics Research Notices. IMRN},
  volume  = {2023},
  number  = {7},
  pages   = {6291--6298},
  year    = {2023},
  doi     = {10.1093/imrn/rnac368},
  note    = {\url{https://doi.org/10.1093/imrn/rnac368}}
}

@misc{BC24,
  author       = {Baur, Karin and Coelho Sim{\~o}es, Raquel},
  title        = {A geometric model for the module category of a string algebra},
  year         = {2024},
  eprint       = {2403.07810v1},
  archiveprefix = {arXiv},
  url          = {https://arxiv.org/abs/2403.07810v1},
  note         = {\url{https://arxiv.org/abs/2403.07810v1}}
}

@article{BR87,
  author  = {Butler, Michael C. R. and Ringel, Claus Michael},
  title   = {Auslander--Reiten sequences with few middle terms and applications to string algebras},
  journal = {Communications in Algebra},
  volume  = {15},
  number  = {1--2},
  pages   = {145--179},
  year    = {1987},
  doi     = {10.1080/00927878708823416},
  note    = {\url{https://doi.org/10.1080/00927878708823416}}
}

@article{C26,
  author  = {Chang, Wen},
  title   = {Geometric models for the algebraic hearts in the derived category of a gentle algebra},
  journal = {Israel Journal of Mathematics},
  volume  = {271},
  pages   = {941--1000},
  year    = {2026},
  doi     = {10.1007/s11856-025-2801-7},
  note    = {\url{https://doi.org/10.1007/s11856-025-2801-7}}
}

@article{CSZ18,
  author  = {Chen, Xiao-Wu and Shen, Dawei and Zhou, Guodong},
  title   = {The {G}orenstein-projective modules over a monomial algebra},
  journal = {Proceedings of the Royal Society of Edinburgh Section A: Mathematics},
  volume  = {148},
  number  = {6},
  pages   = {1115--1134},
  year    = {2018},
  doi     = {10.1017/S0308210518000185},
  note    = {\url{https://doi.org/10.1017/S0308210518000185}}
}

@article{CB89,
  author  = {Crawley-Boevey, William},
  title   = {Maps between representations of zero-relation algebras},
  journal = {Journal of Algebra},
  volume  = {126},
  number  = {2},
  pages   = {259--263},
  year    = {1989},
  doi     = {10.1016/0021-8693(89)90304-9},
  note    = {\url{https://doi.org/10.1016/0021-8693(89)90304-9}}
}

@article{CL2019,
  title={{C}ohen-{M}acaulay {A}uslander algebras of gentle algebras},
  author={Chen, Xinhong  and  Lu, Ming},
  journal={Commun. Algebra},
  fjournal={Communications in Algebra},
  volume={47},
  number={9},
  pages={3597--3613},
  year={2019},
  note={\href{https://doi.org/10.1080/00927872.2019.1570225}{DOI:10.1080/00927872.2019.1570225}},
}

@article{CL2017,
  title={{C}ohen-{M}acaulay {A}uslander algebras of skewed-gentle algebras},
  author={Chen, Xinhong and Lu, Ming},
  journal={Commun. Algebra},
  fjournal={Communications in Algebra},
  volume={45},
  number={2},
  pages={849--865},
  year={2017},
  publisher={Taylor \& Francis},
  note={\href{https://doi.org/10.1080/00927872.2016.1175601}{DOI:10.1080/00927872.2016.1175601}
        (preprint in 2015, \href{https://arxiv.org/abs/1502.03948}{arXiv:1502.03948})},
}

@misc{H21,
  author       = {Hafezi, Rasool},
  title        = {When stable {C}ohen--{M}acaulay {A}uslander algebra is semisimple},
  year         = {2021},
  eprint       = {2109.00467v2},
  archiveprefix = {arXiv},
  url          = {https://arxiv.org/abs/2109.00467v2},
  note         = {\url{https://arxiv.org/abs/2109.00467v2}}
}

@article{K15,
  author  = {Kalck, Martin},
  title   = {Singularity categories of gentle algebras},
  journal = {Bulletin of the London Mathematical Society},
  volume  = {47},
  number  = {1},
  pages   = {65--74},
  year    = {2015},
  doi     = {10.1112/blms/bdu093},
  note    = {\url{https://doi.org/10.1112/blms/bdu093}}
}

@misc{LZH,
  author = {Liu, Yu-Zhe and Zhang, Yafeng and Huang, Zhaoyong},
  title  = {{G}orenstein projective support {$\tau$}-tilting modules over gentle algebras},
  year   = {2022},
  howpublished = {Preprint},
  note   = {\url{http://maths.nju.edu.cn/~huangzy/gtautilt.pdf}},
  url    = {http://maths.nju.edu.cn/~huangzy/gtautilt.pdf}
}

@article{LZ24,
  author  = {Liu, Yu-Zhe and Zhang, Chao},
  title   = {On string algebras and the {C}ohen--{M}acaulay {A}uslander algebras},
  journal = {Applied Categorical Structures},
  volume  = {32},
  pages   = {Paper No. 17},
  year    = {2024},
  doi     = {10.1007/s10485-024-09779-8},
  note    = {\url{https://doi.org/10.1007/s10485-024-09779-8}},
  addendum = {Numbered statements are cited from arXiv:2303.06645v1}
}

@misc{LZhou24,
  author       = {Liu, Yu-Zhe and Zhou, Panyue},
  title        = {On endomorphism algebras of string almost gentle algebras},
  year         = {2024},
  eprint       = {2411.03690v1},
  archiveprefix = {arXiv},
  url          = {https://arxiv.org/abs/2411.03690v1},
  note         = {\url{https://arxiv.org/abs/2411.03690v1}}
}

@article{PPP19,
  author  = {Palu, Yann and Pilaud, Vincent and Plamondon, Pierre-Guy},
  title   = {Non-kissing and non-crossing complexes for locally gentle algebras},
  journal = {Journal of Combinatorial Algebra},
  volume  = {3},
  number  = {4},
  pages   = {401--438},
  year    = {2019},
  doi     = {10.4171/JCA/35},
  note    = {\url{https://doi.org/10.4171/JCA/35}}
}

@article{SZ03,
  author  = {Schr{\"o}er, Jan and Zimmermann, Alexander},
  title   = {Stable endomorphism algebras of modules over special biserial algebras},
  journal = {Mathematische Zeitschrift},
  volume  = {244},
  number  = {3},
  pages   = {515--530},
  year    = {2003},
  doi     = {10.1007/s00209-003-0492-4},
  note    = {\url{https://doi.org/10.1007/s00209-003-0492-4}}
}

@article{XZ24,
  author  = {Xie, Zongzhen and Zhang, Xiaojin},
  title   = {A bijection theorem for {G}orenstein projective {$\tau$}-tilting modules},
  journal = {Journal of Algebra and Its Applications},
  volume  = {23},
  number  = {10},
  pages   = {Paper No. 2450168},
  year    = {2024},
  doi     = {10.1142/S0219498824501688},
  note    = {\url{https://doi.org/10.1142/S0219498824501688}}
}

@misc{XZ26,
  author       = {Xin, Zheng and Zhang, Lingchun},
  title        = {On geometric models of string algebras: uniqueness of surfaces and existence of red punctures},
  year         = {2026},
  eprint       = {2608.14360v1},
  archiveprefix = {arXiv},
  url          = {https://arxiv.org/abs/2608.14360v1},
  note         = {\url{https://arxiv.org/abs/2608.14360v1}}
}

\end{document}